\documentclass[10pt]{amsart} 
\usepackage{amscd,amssymb}
\usepackage{enumerate} 
\usepackage[all]{xypic} 
\usepackage{multirow}
\usepackage{hhline}
\usepackage{verbatim} 
\usepackage{mathdots} 
\usepackage{comment} 
\usepackage{tikz}
\usepackage{appendix}
\usepackage{stmaryrd}

\usetikzlibrary{decorations.markings}

\theoremstyle{plain} 
\newtheorem{Thm}[equation]{Theorem}
\newtheorem*{Thm*}{Theorem} 
\newtheorem{Cor}[equation]{Corollary}

\newtheorem{Lem}[equation]{Lemma}

\theoremstyle{definition} 
\newtheorem*{Rmk}{Remark}

\numberwithin{equation}{section}

\newcounter{myenum}
\newenvironment{enum}{%
  \begin{list}{\arabic{myenum}.}%
    {\setlength{\leftmargin}{20pt}}%
     \setlength{\labelwidth}{0pt}
     \setlength{\itemindent}{0.5em}
     \setlength{\labelsep}{0.5em}
     \setlength{\itemsep}{0.2em}
     \usecounter{myenum}}%
  {\end{list}} 
  
\newcounter{myenumsec}
\newenvironment{enumsec}{%
  \begin{list}{\S\arabic{myenumsec}.}%
    {\setlength{\leftmargin}{13pt}}%
     \setlength{\labelwidth}{0pt}
     \setlength{\itemindent}{0.5em}
     \setlength{\labelsep}{0.5em}
     \setlength{\itemsep}{0.5em}
     \usecounter{myenumsec}}%
  {\end{list}}

\newcommand{\C}{\mathbb C}
 
\newcommand{\Q}{\mathbb{Q}}

\newcommand{\Z}{\mathbb{Z}}
\newcommand{\R}{\mathbb{R}}

\newcommand{\e}{\mathbf{e}}
\newcommand{\f}{\mathbf{f}}

\newcommand{\la}{\langle}
\newcommand{\ra}{\rangle}

\newcommand{\GL}{\operatorname{GL}}
\newcommand{\SL}{\operatorname{SL}}
\newcommand{\Sp}{\operatorname{Sp}}
\newcommand{\SO}{\operatorname{SO}}
\newcommand{\Spt}{\widetilde{\operatorname{Sp}}}

\newcommand{\chit}{\widetilde{\chi}}

\newcommand{\End}{\operatorname{End}}
\newcommand{\supp}{\operatorname{supp}}
\newcommand{\Ind}{\operatorname{Ind}}
\newcommand{\Hom}{\operatorname{Hom}}

\newcommand{\vol}{\operatorname{vol}}

\newcommand{\It}{\widetilde{I}}
\newcommand{\Jt}{\widetilde{J}}
\newcommand{\Kt}{\widetilde{K}}
\newcommand{\Bt}{\widetilde{B}}
\newcommand{\Tt}{\widetilde{T}}
\newcommand{\Pt}{\widetilde{P}}
\newcommand{\Mt}{\widetilde{M}}

\newcommand{\xt}{\widetilde{x}}
\newcommand{\htt}{\widetilde{h}}
\newcommand{\wt}{\widetilde{w}}
\renewcommand{\tt}{\widetilde{t}}
\newcommand{\uu}{\widetilde{u}}

\newcommand{\G}{\mathcal{G}}
\renewcommand{\SS}{\mathcal{S}}
\newcommand{\LL}{\mathcal{L}}

\newcommand{\Ok}{\mathcal{O}}

\newcommand{\tr}{\operatorname{tr}}

\newcommand{\HH}[3]{\mathcal{H}(#1 \! \sslash \! #2; #3)}

\title[Hecke Algebra Correspondences]{Hecke Algebra Correspondences\\
  for the metaplectic group} 

\author{Shuichiro Takeda}
\address{Mathematics Dept, Univ of Missouri-Columbia, 202
  Math Sciences Building, Columbia, MO, 65211}
\email{takedas@missouri.edu} 

\author{Aaron Wood}
\address{Mathematics Dept, Univ of Missouri-Columbia, 202
  Math Sciences Building, Columbia, MO, 65211}
\email{woodad@missouri.edu} 

\begin{document}

\maketitle

\begin{abstract} 
Over a $p$-adic field of odd residual characteristic, Gan and Savin proved a
correspondence between the Bernstein components
of the even and odd Weil representations of the metaplectic group and the components
of the trivial representation of the equal rank odd orthogonal groups. In
this paper, we extend their result to the case of even residual characteristic.
\end{abstract}


\section*{Introduction}


Fix a nonarchimedean local field $k$ of residue characteristic $p$ and
characteristic different from $2$.  Let
$W$ be a non-degenerate symplectic space over $k$ of dimension $2n$
and $\Spt(W)$ the $2$-fold metaplectic cover of $\Sp(W)$.  For an additive
character $\psi$ of $k$, let $\omega_\psi$ be 
the Weil representation of $\Spt(W)$, which decomposes into its even and
odd constituents, $\omega_\psi = \omega_\psi^+ \oplus \omega_\psi^-$.
In the category of genuine, smooth representations of
$\Spt(W)$, let $\G_\psi^\pm$ be the Bernstein component
containing $\omega_\psi^\pm$.

Consider quadratic spaces $V^\pm$ of
dimension $2n+1$ with trivial discriminant, where $V^+$ has the
trivial Hasse invariant and $V^-$ the non-trivial one. Then
$\SO(V^+)$ is the split adjoint group of type B$_ n$ and $\SO(V^-)$ is its unique non-split
inner form. In the category of smooth representations of
$\SO(V^\pm)$, let $\SS^\pm_0$ be the Bernstein component containing the trivial
representation of $\SO(V^\pm)$.  

Let $\epsilon$ be + or --.  In \cite{GS2}, Gan and Savin proved an equivalence of categories between $\G_\psi^\epsilon$ and 
$\SS_0^\epsilon$ assuming that $p \neq 2$.  
The aim of this paper is to extend their result to the case
of even residual characteristic.  We follow their general strategy of
exploiting minimal types of the Weil representation to define a Hecke algebra $H_\psi^\epsilon$,
showing that the category $\G_\psi^\epsilon$ is equivalent to the category of $H_\psi^\epsilon$-modules,
and giving an isomorphism between $H_\psi^\epsilon$ and the standard 
Iwahori-Hecke algebra of $\SO(V^\epsilon)$.

A key ingredient for extending their result is an analysis of the $K$-types of the Weil representation
in arbitrary residual characteristic which was carried out by Savin and the second-named author in
\cite{SW}.  We also employ the machinery of Bushnell, Henniart, and Kutzko in \cite{BHK} to 
compare the Plancherel measures induced from the respective Hecke algebras.
More explicitly, the layout of the paper is as follows.
\begin{enumsec}

\item 
We introduce notation and summarize some relevant background material.

\item
We describe a minimal type  for an open compact subgroup 
 and compute the corresponding
spherical Hecke algebra 
$H_\psi^\epsilon$.  We give an
isomorphism between $H_\psi^\epsilon$ and the standard Iwahori-Hecke algebra $H^\epsilon$ of 
$\SO(V^\epsilon)$.  We show that the isomorphism $H_\psi^\epsilon \cong H^\epsilon$ is, 
in fact, an isomorphism of 
Hilbert algebras with involution, thus giving a coincidence of induced Plancherel measures under
suitable normalization.  A corollary of this result is that the correspondence of
Hecke algebra modules preserves formal degree.

\item
We prove that the category of $H_\psi^\epsilon$-modules is equivalent to the category $\G_\psi^\epsilon$,
thus giving the desired equivalence of categories $\G_\psi^\epsilon \cong \SS_0^\epsilon$.  From the
theory of Plancherel measures, we deduce that this equivalence preserves the temperedness and 
square-integrability of representations.
\end{enumsec}

\section*{Acknowledgements}

The authors would like to thank Gordan Savin for his insight and suggestions.  In particular, the idea
to use vector-valued Hecke algebras arose in conversations between him and the second-named author
while participating at the Oberwolfach workshop on ``Spherical Spaces and Hecke Algebras.''  The 
second-named author would
like to thank the Oberwolfach Research Institute for Mathematics (MFO) for fostering a stimulating
research environment.  The authors would also like to thank Wee Teck Gan for answering several questions regarding \cite{GS2}.

The first-named author was partially supported by NSF grant
DMS-1215419. 

Finally, the authors would like to thank the anonymous referee to make
various useful suggestions.

\section{Preliminaries}

Throughout the paper, $k$ is a nonarchimedean local field with residual characteristic $p$;
we allow for arbitrary residual characteristic but assume that the characteristic of $k$ is
different from $2$.  Let $\Ok$ be the ring of integers and $\varpi$ a chosen uniformizer.
Denote by $q$ the
cardinality of the residue field and by $e$ the valuation of 2
in $k$.  If $p = 2$, then $e$ is the ramification index of 2; 
otherwise $e = 0$.
Let $\psi$ be a non-trivial additive character of $k$; for convenience, we
assume that $\psi$ has conductor $2e$,
i.e., that $4\Ok$ is the largest additive subgroup of $\Ok$ on which $\psi$
acts trivially.

For a vector space $V$ over $k$, we denote by $S(V)$ the Schwartz space of smooth, 
compactly supported, $\C$-valued functions on $V$.  We denote the subspaces 
of even and odd functions in $S(V)$ by $S(V)^+$ and $S(V)^-$, respectively.

\subsection{The symplectic group $\Sp(W)$}

Let $W$ be a non-degenerate symplectic space over $k$ of dimension
$2n$ with basis $\{\e_1,\dots,\e_n,\f_1,\dots,\f_n\}$, where
$\la\e_i,\e_j\ra=0=\la\f_i,\f_j\ra$ for all $i, j$ and
$\la\e_i,\f_j\ra=\delta_{i,j}$.  The symplectic group $\Sp(W)$ is the 
group of invertible transformations of $W$ which preserve
the symplectic form.  The decomposition $W = X + Y$, where $X$ is the span of the $\e_i$
and $Y$ is the span of the $\f_i$, is a polarization of $W$.  

Let $W_\C$ be the $\C$-span of the symplectic basis and $\mathfrak{sp}(W_\C)$ the 
symplectic Lie algebra, consisting of endomorphisms 
$T: W_\C \to W_\C$ such that $\la Tu, v \ra + \la u, Tv \ra = 0$ for all $u,v \in W_\C$.  
Let $\mathfrak{h}$ be the diagonal Cartan subalgebra relative to the symplectic basis and
$\mathfrak{h}^\ast = \Hom_\C ( \mathfrak{h}, \C )$ its linear dual.
The roots of $\mathfrak{sp}(W_\C)$ form a root system of type C$_n$, defined by
\[
	\Sigma = 
		\{\pm \epsilon_i \pm \epsilon_j : 1 \leq i < j \leq n\} 
			\cup 
		\{\pm 2\epsilon_i: 1\leq i\leq n\} 
			\subset \mathfrak{h}^\ast,
\] 
where $\epsilon_i: \mathfrak{h} \to \C$ is given by
\[
	H = \begin{pmatrix} a & \\ & -a \end{pmatrix} \mapsto \epsilon_i(H) = a_i.
\]
We take $\Delta=\{\alpha_1,\dots,\alpha_n\}$ as the set of
simple roots, where $\alpha_n = 2\epsilon_n$ and 
$\alpha_i=\epsilon_i-\epsilon_{i+1}$ otherwise.  This choice of simple roots
decomposes $\Sigma$ into positive roots $\Sigma^+$ and negative roots $\Sigma^-$.

Each root $\alpha \in \mathfrak{h}^\ast$ has a corresponding coroot $\check{\alpha} \in \mathfrak{h}$ such that $\alpha(\check{\alpha}) = 2$; the coroots form a root system of type B$_n$.  Denote by 
$\mathfrak{h}_\R$ the real span of the coroots.

Let $\Sigma_a = \{\alpha+m : \alpha \in \Sigma, m \in \Z\}$ be the set of affine roots, where $\alpha + m$ is the affine functional on $\mathfrak{h}_\R$ given by $(\alpha + m)(H) = \alpha(H) + m$.  
We take $\Delta_a=\Delta\cup\{\alpha_0\}$ to be the set of simple affine roots, where 
$\alpha_0 = -2\epsilon_1 + 1$.

For each affine root $\alpha + m$, define $s_{\alpha+m}$ to be the 
reflection across the affine hyperplane $P_{\alpha+m} = \{ x \in \mathfrak{h}_\R : \alpha(x)+m = 0 \}$.  
We write $s_i = s_{\alpha_i}$ for the simple
affine reflections across the affine hyperplanes $P_i = P_{\alpha_i}$.  
The affine space $\mathfrak{h}_\R$ is an apartment for $\Sp(W)$.  The chambers of the apartment 
are the connected components of $\mathfrak{h}_\R \smallsetminus \bigcup P_{\alpha+m}$.  For $n = 2$, 
the root system, coroot system, and apartment are sketched below.

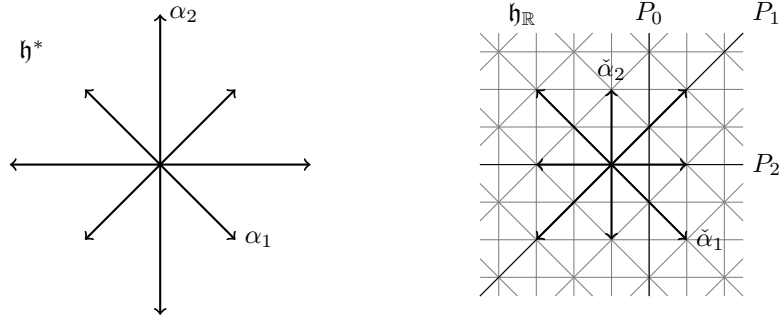
\begin{center}
\begin{tikzpicture}
	\node [right] at (-2,1.5) {$\mathfrak{h}^\ast$};
	\draw [<->, thick] (-1,-1) -- (1,1);
	\draw [<->, thick] (-2,0) -- (2,0);
	\draw [<->, thick] (-1,1) -- (1,-1);
		\node [right] at (1,-1) {$\alpha_1$};
	\draw [<->, thick] (0,-2) -- (0,2);
		\node [right] at (0,2) {$\alpha_2$};
	
	\foreach \a in {4.5,5,5.5,6,6.5,7,7.5}
		{	\draw [gray] (\a,-1.75) -- (\a,1.75);		}
	\foreach \a in {-1.5,-1,-.5,0,.5,1,1.5}
		{	\draw [gray] (4.25,\a) -- (7.75,\a);		}
	\foreach \a in {0,1,2,3}
		{	\draw [gray] (4.25,-1.75+\a) -- (7.75-\a,1.75);		}
	\foreach \a in {1,2,3}
		{	\draw [gray] (4.25+\a,-1.75) -- (7.75,1.75-\a);		}
	\foreach \a in {0,1,2,3}
		{	\draw [gray] (4.25,1.75-\a) -- (7.75-\a,-1.75);		}
	\foreach \a in {1,2,3}
		{	\draw [gray] (4.25+\a,1.75) -- (7.75,-1.75+\a);		}
	\draw (4.25,-1.75) -- (7.75,1.75);
		\node [above right] at (7.75,1.75) {$P_1$};
	\draw (4.25,0) -- (7.75,0);
		\node [right] at (7.75,0) {$P_2$};
	\draw (6.5,-1.75) -- (6.5,1.75);
		\node [above] at (6.5,1.75) {$P_0$};
	\node [right] at (4.5,2) {$\mathfrak{h}_\R$};
	\draw [<->, thick] (5,-1) -- (7,1);
	\draw [<->, thick] (5,0) -- (7,0);
	\draw [<->, thick] (5,1) -- (7,-1);
		\node [right] at (7,-1) {$\check{\alpha}_1$};
	\draw [<->, thick] (6,-1) -- (6,1);
		\node [above] at (6,1) {$\check{\alpha}_2$};
\end{tikzpicture}
\end{center}

The Weyl group $\Omega$ is the group generated by the simple reflections $s_1, \dots, s_n$.  The affine
Weyl group $\Omega_a$ is the group generated by the affine simple reflections $s_0, s_1, \dots, s_n$;  
it is the semi-direct product $\Omega_a = D \Omega$ of a translation group $D$ and the 
Weyl group.  Both $\Omega$ and $\Omega_a$ are Coxeter groups whose braid relations are given according 
to the following Coxeter diagram.

\begin{center}
\begin{tikzpicture}[scale=.45]
    \draw (0,-.1) --(2,-.1);
    \draw (0,.1) --(2,.1);
    \draw (2,0) --(5,0);
    \draw (7,0) --(8,0);
    \draw (8,-.1) --(10,-.1);
    \draw (8,.1) --(10,.1);
    \draw[fill] (0,0) circle(5pt);
    \draw[fill] (2,0) circle(5pt);
    \draw[fill] (4,0) circle(5pt);
    \draw[fill] (8,0) circle(5pt);
    \draw[fill] (10,0) circle(5pt);
    \node[below] at (0,-.1) {$s_0$};
    \node[below] at (2,-.1) {$s_1$};
    \node[below] at (4,-.1) {$s_2$};
    \node[below] at (10,-.1) {$s_n$};
    \node at (6,0) {$\cdots$};
\end{tikzpicture}
\end{center}

For each root $\alpha\in \Sigma$, 
we fix a map
	$
		\Phi_\alpha:\SL_2(k)\rightarrow\Sp(W)
	$
such that the images of the unipotent upper and lower triangular
matrices in $\SL_2(k)$ are the root subgroups of $\Sp(W)$ corresponding
to $\alpha$ and $-\alpha$,
respectively.  For $\alpha+m \in \Sigma_a$, we
define the map $\Phi_{\alpha+m}:\SL_2(k)\rightarrow\Sp(W)$ by
	\[
		\Phi_{\alpha+m}\begin{pmatrix}a&b\\c&d\end{pmatrix} = 
		\Phi_\alpha\begin{pmatrix}a&\varpi^mb\\\varpi^{-m}c&d\end{pmatrix};
	\]
we write
	\begin{align*}
		x_{\alpha+m}(t) 
			&
			= \Phi_{\alpha+m} \begin{pmatrix} 1 & t \\ 0 & 1 \end{pmatrix} 
			= x_\alpha(\varpi^m t)
			\quad \quad (t \in k),
		\\
		w_{\alpha+m}(t) 
			&
			= \Phi_{\alpha+m}\begin{pmatrix}0&t\\-t^{-1}&0\end{pmatrix} 
			= w_\alpha(\varpi^m t)
			\quad \quad (t \in k^\times),
		\\	
		h_{\alpha+m}(t) 
			&
			= \Phi_{\alpha+m}\begin{pmatrix}t&0\\0&t^{-1}\end{pmatrix} 
			= h_\alpha(t)
			\quad \quad (t \in k^\times).
	\end{align*}
		
We take the element $w_{\alpha_i}(1)$ as a representative in $\Sp(W)$ of the simple affine 
reflection $s_i$.  We will frequently use the same notation to refer to an element
$w = s_{i_1} \cdots s_{i_r}$ in $\Omega_a$ and its representative $w = w_{\alpha_{i_1}}(1) \cdots w_{\alpha_{i_r}}(1)$ in $\Sp(W)$.

\subsection{Open compact subgroups of $\Sp(W)$}

For $0 \le i \le n$, we define the lattice 
\[
\LL_i=\operatorname{Span}_{\Ok}\{\e_1,\dots,\e_n, \varpi\f_1, \dots,
\varpi\f_i, \f_{i+1},\dots,\f_n\}.
\]
The stabilizer
$
K_i=\{g\in\Sp(W):g\LL_i\subseteq \LL_i\}
$
is a maximal open compact subgroup of $\Sp(W)$; it is the group generated
by those $\Phi_{\alpha_j}(\Ok)$ for which $j \neq i$; it is also the stabilizer
of the point $z_i$ in the apartment, where
\[
	z_0 = (0, 0, \ldots, 0),
		\quad
	z_1 = (\tfrac{1}{2}, 0, \ldots, 0), 
		\quad \ldots, \quad
	z_n = (\tfrac{1}{2}, \tfrac{1}{2}, \ldots, \tfrac{1}{2}).
\]
In this way, each $K_i$ corresponds to the vertex $z_i$ in the `standard' apartment;
hence, every maximal open compact subgroup of $\Sp(W)$ is conjugate to one of the $K_i$, 
cf.\ \cite[\S 3.2]{Ti}.

The intersection of $K_0, \dots, K_n$ 
is an Iwahori subgroup $I$.  The unipotent radical of $I$ is generated by the simple affine root groups $\Phi_{\alpha_i}(\SL_2(\Ok))$.  In this way, the Iwahori subgroup $I$ corresponds to the chamber 
in the apartment which is bounded by the hyperplanes $P_0, \dots, P_n$; in addition, the vertices of the
chamber are precisely $z_0, \ldots, z_n$.  The rank 2 picture is as follows.

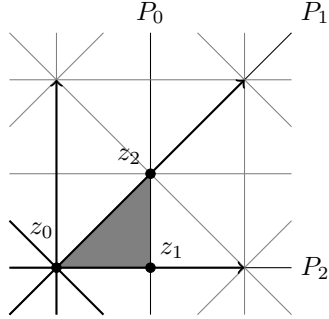
\begin{center}
\begin{tikzpicture}[scale=2.5]
	\draw [fill, gray] (0,0) -- (.5,0) -- (.5,.5) -- (0,0);
	\draw (-.25,0) -- (1.25,0);
	\draw (-.25,-.25) -- (1.25,1.25);
	\draw [gray] (0,-.25) -- (0,1.25);
	\draw (.5,-.25) -- (.5,1.25);
	\draw [gray] (-.25,1.25) -- (1.25,-.25);
	\draw [gray] (.75,1.25) -- (1.25,.75);
	\draw [gray] (-.25,.75) -- (.25,1.25);
	\draw [gray] (.75,-.25) -- (1.25,.25);
	\draw [gray] (1,-.25) -- (1,1.25);
	\draw [gray] (-.25,.5) -- (1.25,.5);
	\draw [gray] (-.25,1) -- (1.25,1);
	\draw [->, thick] (-.25,0) -- (1,0);
	\draw [->, thick] (-.25,-.25) -- (1,1);
	\draw [->, thick] (0,-.25) -- (0,1);
	\draw [thick] (-.25,.25) -- (.25,-.25);
	\draw [fill] (0,0) circle(.025);
		\node at (-.08,.2) {$z_0$};
	\draw [fill] (.5,0) circle(.025);
		\node [above right] at (.5,0) {$z_1$};
	\draw [fill] (.5,.5) circle(.025);
		\node [above left] at (.5,.5) {$z_2$};
	\node [right] at (1.25,0) {$P_2$};
	\node [above right] at (1.25,1.25) {$P_1$};
	\node [above] at (.5,1.25) {$P_0$};
\end{tikzpicture}
\end{center}

The double cosets in 
$I\,\backslash\!\Sp(W)\slash I$ are parameterized by the affine Weyl group; namely, each 
$I$-double coset is of the form $IwI$ for some $w \in \Omega_a$.  The number of $I$-single cosets
in $IwI$ is 
\[
[IwI: I]=q^{\ell(w)},
\]
where $\ell$ is the length function on $\Omega_a$.

\subsection{Metaplectic group and the Weil representation}
\label{S:Weil}

For a polarization $W = X + Y$, the Schwartz space $S(Y)$ realizes the unique (up to isomorphism)
representation $\rho_\psi$ of the Heisenberg group with the central
character $\psi$. Via the action of $\Sp(W)$ on the Heisenberg group,
$\rho_\psi$ gives a projective representation of $\Sp(W)$ which lifts to
a linear representation $\omega_\psi$, called the Weil representation, 
of the central extension $\Sp'(W)$ of $\Sp(W)$ given by 
\[
1\rightarrow\C^\times\rightarrow\Sp'(W)\rightarrow\Sp(W)\rightarrow 1.
\]
It is a theorem of Weil that the derived group of $\Sp'(W)$ is a 2-fold
cover $\Spt(W)$ of $\Sp(W)$ and that $\omega_\psi$ is a faithful representation of $\Spt(W)$, 
cf.\ \cite[IV.42-43]{We}, \cite[2.II.1]{MVW}.
 
For a subgroup $H\subseteq \Sp(W)$, we denote its preimage in $\Spt(W)$ by
$\widetilde{H}$.  For each root $\alpha$,  the element $x_\alpha(t)$
canonically lifts to an element $\xt_\alpha(t)$ in $\Spt(W)$. We may therefore define lifts of
$w_\alpha(t)$ and $h_\alpha(t)$ via the formulas
\begin{align*}
\wt_\alpha(t)&=\xt_\alpha(t)\xt_{-\alpha}(-t^{-1})\xt_\alpha(t),\\
\htt_\alpha(t)&=\wt_\alpha(t)\wt_\alpha(-1).
\end{align*}

We will take $\wt_{\alpha_i}(1)$ for a representative in $\Spt(W)$ of the affine simple reflection $s_i$.
We will continue to abuse notation when referring to an element 
of $\Omega_a$ or its representatives
in either $\Sp(W)$ or $\Spt(W)$.

Explicitly, the Weil representation on $S(Y)$ is given by
\begin{align*}
	\big[ \xt(a) \phi \big](y)	&	=	\psi( {}^t\!y a y ) \phi(y),					\\
	\big[ \htt(a) \phi \big](y)	&	=	\beta_a |\det a|^{1/2} \phi( {}^t\!y ),		\\
	\big[ \wt \phi \big](y)		&	=	\gamma_1 \widehat{\phi}(y).
\end{align*}
Here, $\xt(a)$, $\htt(a)$, and $\wt$ are respective lifts of
\[
	\begin{pmatrix} 1 & a \\ 0 & 1 \end{pmatrix},
		\quad \quad
	\begin{pmatrix} a & 0 \\ 0 & {}^t a^{-1} \end{pmatrix},
		\quad \text{ and } \quad
	\begin{pmatrix} 0 & 1 \\ -1 & 0 \end{pmatrix};
\]
the Fourier transform of $\phi$ is defined by
\[
	\widehat\phi(y) = \int_Y \psi(2\, {}^t\!u y) \phi(u) du;
\]
and $\beta_a$ and $\gamma_1$ are specific 8th roots of unity, whose precise value plays no role in the current investigation.

Under the Weil representation $\omega_\psi$, the (positive) root groups act as follows:
\begin{align*}
	\big[ \xt_{\epsilon_i - \epsilon_j}(t) \phi \big] (y) 	&	= \phi(y + t y_i \f_j),			\\
	\big[ \xt_{\epsilon_i + \epsilon_j}(t) \phi \big] (y)	&	= \psi(2 t y_i y_j) \phi(y),	\\
	\big[ \xt_{2\epsilon_i}(t) \phi \big] (y)					&	= \psi( t y_i^2 ) \phi(y).
\end{align*}

\subsection{Minimal types of the Weil representation}

Realized as a representation of $S(Y)$, the Weil representation
$\omega_\psi$ decomposes into the sum of even and odd functions, $\omega_\psi^+\oplus\omega_\psi^-$.
We consider the lattices $L_i = \LL_i \cap Y$.  As computed in \cite{SW}, $\Kt_i$ acts on 
$
\tau_i = S( L_0 / 2L_i ),
$
viewed naturally as a subspace of $S(Y)$. 
The space $\tau_0$ consists entirely of even functions and is an irreducible $\Kt_0$-module.  
Otherwise, as a $\Kt_i$-module, $\tau_i$ 
decomposes as $\tau_i^+ \oplus \tau_i^-$.  Each $\tau_i^\pm$ 
admits a tensor product structure,
\[
\tau_i^\pm = S( \Ok \f_1 / 2 \varpi \Ok \f_1 )^\pm \otimes
			\cdots \otimes
			S( \Ok \f_i / 2 \varpi \Ok \f_i )^\pm \otimes
			S( \Ok \f_{i+1} / 2 \Ok \f_{i+1} ) \otimes
			\cdots \otimes
			S( \Ok \f_n / 2 \Ok \f_n ),
\]
hence the dimension of $\tau_i^\pm$ is $\tfrac{1}{2}q^{en}(q^i \pm 1)$.   We note that
$\tau_i\subseteq\tau_{i+1}$ for $0 \leq i < n$.

We also note that, for $1 \le i \le n-1$, the simple affine reflection $s_i$ essentially acts on $\phi \in S(Y)$ by interchanging the $i$th and $(i+1)$th components; the reflection $s_n$ acts essentially via Fourier transform on the $n$th component; the reflection $s_0$ acts on the first component $\phi_1$ of $\phi$ as
\[
	\big[s_0 \phi_1 \big](y_1) = c\widehat\phi_1(\varpi^{-1}y_1),
\]
for some constant $c$.  

Lastly, we note that the Iwahori group $\It = \Kt_0 \cap \dots \cap \Kt_n$ is contained in each $\Kt_i$, hence it preserves each of the minimal types $\tau_i^\pm$.  Similarly, the group $\Jt = \Kt_1 \cap \dots \cap \Kt_n$ is contained in $\Kt_i$ for $i \ne 0$, so it preserves the minimal type $\tau_i^\pm$ for $i \ne 0$.

We record these observations in the following lemma.

\begin{Lem}
\label{L:inflate}
The Iwahori group $\It$ preserves $\tau_i^\pm$ for $0 \le i \le n$, and the group 
$\Jt$ preserves $\tau_i^\pm$ for $1 \le i \le n$.  Moreover,
\begin{enum}

\item
the elements $s_1, \dots, s_n$ preserve $\tau_0$ while $s_0$ inflates $\tau_0$ to $\tau_1^+$;

\item
the elements $s_2, \dots, s_n$ preserve $\tau_1^-$ while $s_1$ inflates $\tau_1^-$ to $\tau_2^-$.

\end{enum}
\end{Lem}

\subsection{Spherical Hecke algebras}\label{S:Hecke}

We summarize some generalities on Hecke algebras, most of which may be found in \cite{GS2}.

Let $G$ be a totally disconnected topological group and $K\subseteq G$
an open compact subgroup; fix a Haar measure $dg$ on $G$. 
For an irreducible, finite-dimensional representation
$(\sigma, V_\sigma)$ of $K$, let $(\sigma^*, V_\sigma^*)$ be its contragredient
representation and define the $\sigma$-spherical Hecke algebra by
	\[
		\HH{G}{K}{\sigma}
			= 	
				\left\{ 
					f : G \to \End(V_\sigma^*) : 
					\begin{array}{l}
						f \text{ is smooth and compactly supported},	\\
						f(k_1 g k_2) = \sigma^*(k_1) f(g) \sigma^*(k_2),
							\text{ for }k_i \in K, g \in G 
					\end{array}
				\right\};
	\]
it is an algebra under convolution with an identity element which we denote $1_\sigma$. 

For a smooth
representation $(\pi, V_\pi)$ of $G$, consider the space $(V_\pi\otimes V_\sigma^*)^K$ of
$K$-fixed vectors in $V_\pi\otimes V_\sigma^*$; this space admits a natural action of $\HH{G}{K}{\sigma}$ by
\[
\pi(f)(v\otimes e)=\int_G\pi(g)v\otimes f(g)e\;dg,
\]
where $v\otimes e\in(V_\pi\otimes V_\sigma^*)^K$ and $f\in \HH{G}{K}{\sigma}$.

Let $\Gamma$ be an open compact subgroup of $G$ containing $K$; assume that the index $[\Gamma:K]$
is finite.  We consider $\HH{\Gamma}{K}{\sigma}$ as a finite-dimensional subalgebra of 
$\HH{G}{K}{\sigma}$ via
\[
\HH{\Gamma}{K}{\sigma}
	= \{f\in \HH{G}{K}{\sigma}: \supp(f)\subseteq\Gamma\}.
\]
We have a natural isomorphism
$
\xymatrix@1{
L: \HH{\Gamma}{K}{\sigma}\ar[r]^-{\sim}&\End_{\Gamma}(\Ind_K^\Gamma(\sigma^*))
}
$
given by
\[
(L(f)\phi)(g)=\int_\Gamma f(h)\phi(h^{-1}g)\,dh
\]
for $f\in \HH{\Gamma}{K}{\sigma}$, $\phi\in \Ind_K^\Gamma(\sigma^*)$, and
$g\in\Gamma$. 

Suppose that $(\pi, V_\pi)$ is an irreducible, smooth, finite dimensional
representation of $\Gamma$ such that $(V_\pi\otimes V_\sigma^*)^K\neq 0$. Then
$(V_\pi \otimes V_\sigma^*)^K$ is a simple $\HH{\Gamma}{K}{\sigma}$-module via the action 
of $\HH{G}{K}{\sigma}$.  Now assume $\HH{\Gamma}{K}{\sigma}$ is commutative. Then 
$(V_\pi\otimes V_\sigma^*)^K$ is one dimensional, and the action of $\HH{\Gamma}{K}{\sigma}$ 
factors through a maximal ideal $\mathfrak{m}\subseteq \HH{\Gamma}{K}{\sigma}$. Moreover
\[
\Ind_K^\Gamma(\sigma^*)\slash \left(L(\mathfrak{m})\cdot
\Ind_K^\Gamma(\sigma^*)\right)\cong \pi^*.
\]
Therefore if $(\pi_1, V_1),\dots,(\pi_l, V_r)$ are the irreducible
representations (up to isomorphism) of $\Gamma$ such that
$(V_i\otimes V_\sigma^*)^K\neq 0$, then we have
\[
\Ind_K^\Gamma(\sigma^*)\cong \pi_1^*\oplus\cdots\oplus \pi_r^*.
\]
For each $f\in \HH{\Gamma}{K}{\sigma}$, the trace of $L(f)$ is
$
\lambda_1 d_1 + \cdots + \lambda_r d_r,
$
where $d_i=\dim V_i$ and $\lambda_i=\pi_i(f)$.  If $f$ is not
supported on $K$, then the trace of $L(f)$ is $0$.  The case of $r=2$ is summarized 
by the following lemma.

\begin{Lem}\label{L:Hecke}
Suppose that $\dim \HH{\Gamma}{K}{\sigma}=2$ with $T \in \HH{\Gamma}{K}{\sigma}$ not supported
on $K$.  Let $(\pi_i, V_i)$, for $i = 1,2$, be the two irreducible
representations (up to isomorphism) of $\Gamma$ such that $(V_i\otimes
V_\sigma^*)^K\neq 0$.  Write $d_i = \dim V_i$ and $\lambda_i = \pi_i(T)$.  Then
\begin{enum}
\item $\lambda_1 d_1 + \lambda_2 d_2 = 0$;
\item the dimension of $\Ind_K^\Gamma(\sigma^*)$ is $d=d_1+d_2$;
\item the minimal polynomial of $T$ is $(T-\lambda_1)(T-\lambda_2)=0$.
\end{enum}
\end{Lem}

For the groups and representations we will consider (specifically, representations on $\C$-vector spaces of connected reductive $k$-groups and their central extensions), there is additional structure on $\HH{G}{K}{\sigma}$, namely the $*$-operation, 
$f^\ast(g) = \overline{f(g^{-1})}$, and the trace operation, $\tr(f) = f(1)$.  
Following \cite[\S4.1]{BHK}, $\HH{G}{K}{\sigma}$ is a normalized Hilbert algebra with involution 
$f \mapsto f^*$ and scalar product	 
	\[
		[ f_1, f_2 ] = \frac{\vol(K)}{\dim \sigma} \tr (f_1^* f_2).
	\]
This structure yields a Plancherel formula on $\HH{G}{K}{\sigma}$: there is a positive
Borel measure $\mu_\sigma$ on the $C^*$-algebra completion 
$C^*(K,\sigma)$ of 
$\HH{G}{K}{\sigma}$
such that
	\[
		[f, 1_\sigma ] 
			= \int_{\widehat{C^*}(K,\sigma)} 
					\tr \pi(f) \operatorname{d}\!\hat{\mu}_\sigma(\pi).
	\]
Note that $\mu_\sigma$ depends on the chosen Haar measure of $G$.
	
We now consider this situation for two such groups, $G_1$, $G_2$.  For $i = 1,2$, fix an open compact subgroup 
$K_i \subseteq G_i$, an irreducible smooth representation $\sigma_i$ of $K_i$, and a Haar measure
$\mu_i$ of $G_i$. Let $\widehat{\mu}_i$  be the Plancherel measure 
on $\widehat{G}_i$ with respect to the Haar measure $\mu_i$;
following the notation of \cite{BHK} we denote by
$_r \widehat{G}_i$ the support of $\widehat{\mu}_i$.  We write 
$_r \widehat{G}_i(\sigma_i)$ for
the subspace of $_r \widehat{G}_i$ consisting of the representations $\pi$ for which 
$(\pi \otimes \sigma_i^*)^{K_i} \neq 0$.

From \cite[\S5.2]{BHK}, if we have an isomorphism of Hecke algebras
	\[
		\alpha : \HH{G_1}{K_1}{\sigma_1} \to \HH{G_2}{K_2}{\sigma_2}
	\]
such that, for all $f \in \HH{G_1}{K_1}{\sigma_1}$,
\begin{enum}
	\item
	$\alpha(f^*) = \alpha(f)^*$, and
	
	\item
	$\tr(f) = 0$ implies $\tr \big( \alpha(f) \big) = 0$,
\end{enum}
then it is an isomorphism of Hilbert algebras.  We then apply \cite[Cor. C, p.57]{BHK}.  

\begin{Lem}
\label{L:Transfer}
An isomorphism 
	\[
		\alpha: \HH{G_1}{K_1}{\sigma_1} \to \HH{G_2}{K_2}{\sigma_2}
	\]
of Hilbert algebras induces a homeomorphsim
	\[
		\widehat\alpha : {_r} \widehat{G}_2(\sigma_2) \to {_r} \widehat{G}_1(\sigma_1)
	\]
such that
	\[
		\frac{\mu_1(K_1)}{\dim \sigma_1} \widehat{\mu}_1(\widehat\alpha(S))
			= \frac{\mu_2(K_2)}{\dim \sigma_2} \widehat{\mu}_2(S)
	\]
for any Borel subset $S$ of $_r \widehat{G}_2(\sigma_2)$.
\end{Lem}

In the latter sections, we will apply this lemma with $G_1 = \Spt(W)$.  Strictly speaking,
the groups considered in \cite{BHK} are connected, reductive $k$-groups; however, there is
no obstruction in extending this result to the metaplectic group.


\section{Hecke algebra isomorphisms}


In this section, we define our Hecke algebras
$H_\psi^\pm$ of $\Spt(W)$ and
show that they are isomorphic to the affine Hecke algebras 
$H^\pm$ of $\SO(V^\pm)$.


\subsection{Hecke algebra of $\SO(V^+)$}


Let $V^+$ be a quadratic space of dimension $2n+1$ with trivial discriminant and trivial Hasse
invariant; then $\SO(V^+)$ is a split, adjoint, orthogonal 
group of type B$_n$.  Let $I^+$ and $\Omega_a^+$ denote its Iwahori subgroup and affine Weyl 
group, respectively.  The standard Iwahori-Hecke algebra is the set of smooth,
compactly-supported $I^+$\!-bi-invariant functions on $\SO(V^+)$, 
	\[
		H^+ = \HH{\SO(V^+)}{I^+}{\mathbf{1}}.
	\] 
For each $w \in \Omega_a^+$, take $U_w$ to be the characteristic function on the double coset 
$I^+ w I^+$.  The collection $\{U_w\}$ forms a basis of $H^+$ as a vector space.  As an 
algebra, $H^+$ is generated by elements $U_0, \dots, U_n$, and $\sigma$, where $U_i = U_{w_i}$ 
for $w_i$ a simple affine reflection in $\Omega_a^+$, and $\sigma$ is the outer automorphism 
which exchanges the nodes on the Coxeter diagram corresponding to $U_0$ and $U_1$.  The 
quadratic relations for the $U_i$ are
	\[	(U_i + 1) (U_i - q) = 0,		\]
and the braid relations are given by the affine diagram of type B$_n$. 
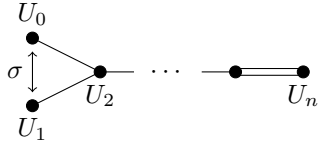
\begin{center}
\begin{tikzpicture}[scale=.45]
    \draw (0,-1) --(2,0);
    \draw (0,1) --(2,0);
    \draw (2,0) --(3,0);
    \draw (5,0) --(6,0);
    \draw (6,-.1) --(8,-.1);
    \draw (6,.1) --(8,.1);
    \draw[fill] (0,1) circle(5pt);
    \draw[fill] (0,-1) circle(5pt);
    \draw[fill] (2,0) circle(5pt);
    \draw[fill] (6,0) circle(5pt);
    \draw[fill] (8,0) circle(5pt);
    \node[above] at (0,1.1) {$U_0$};
    \node[below] at (0,-1.1) {$U_1$};
    \node[below] at (2,-.1) {$U_2$}; 
    \node[below] at (8,-.1) {$U_n$};
    \node at (4,0) {$\cdots$};
    \draw[<->] (0,-.6) --(0,.6);
    \node[left] at (0,0) {$\sigma$};
\end{tikzpicture}
\end{center}
For details, see \cite[\S3]{IM}.

Noting that $\sigma^2 = 1$ and $\sigma U_1 \sigma = U_0$, we see that $U_0$ is abstractly 
unnecessary as a generator.  Hence, $H^+$ is generated by $\sigma$, $U_1, \dots, U_n$ 
subject to the quadratic relations,
	\[	(\sigma + 1)(\sigma - 1) = 0
			\quad \text{ and } \quad
		(U_i + 1)(U_i - q) = 0,	\]
and the braid relations given by the affine diagram of type C$_n$.
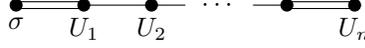
\begin{center}
\begin{tikzpicture}[scale=.45]
    \draw (0,-.1) --(2,-.1);
    \draw (0,.1) --(2,.1);
    \draw (2,0) --(5,0);
    \draw (7,0) --(8,0);
    \draw (8,-.1) --(10,-.1);
    \draw (8,.1) --(10,.1);
    \draw[fill] (0,0) circle(5pt);
    \draw[fill] (2,0) circle(5pt);
    \draw[fill] (4,0) circle(5pt);
    \draw[fill] (8,0) circle(5pt);
    \draw[fill] (10,0) circle(5pt);
    \node[below] at (0,-.1) {$\sigma$};
    \node[below] at (2,-.1) {$U_1$};
    \node[below] at (4,-.1) {$U_2$}; 
    \node[below] at (10,-.1) {$U_n$};
    \node at (6,0) {$\cdots$};
\end{tikzpicture}
\end{center}


\subsection{$\tau_0$-spherical Hecke algebra of $\Spt(W)$}


The restriction of the minimal type $\tau_0$ from $\Kt_0$ to the
Iwahori subgroup $\It$ remains irreducible, as shown in \cite{SW}.  In
this section, we compute the $\tau_0$-spherical Hecke algebra 
	\[
		H_\psi^+ = \HH{\Spt(W)}{\It}{\tau_0}.
	\]

\begin{Thm}\label{T:isomorphism1}
The Hecke algebra $H_\psi^+$ is generated by
invertible elements $T_0, T_1, \dots, T_n$, satisfying the quadratic relations
	\[	(T_0 + 1)(T_0 - 1) = 0
			\quad \text{ and } \quad
		(T_i + 1)(T_i - q) = 0
			\quad \text{ for } i \neq 0,	\] 
and the braid relations of affine diagram of type C$_n$.
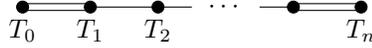
\begin{center}
\begin{tikzpicture}[scale=.45]
    \draw (0,-.1) --(2,-.1);
    \draw (0,.1) --(2,.1);
    \draw (2,0) --(5,0);
    \draw (7,0) --(8,0);
    \draw (8,-.1) --(10,-.1);
    \draw (8,.1) --(10,.1);
    \draw[fill] (0,0) circle(5pt);
    \draw[fill] (2,0) circle(5pt);
    \draw[fill] (4,0) circle(5pt);
    \draw[fill] (8,0) circle(5pt);
    \draw[fill] (10,0) circle(5pt);
    \node[below] at (0,-.1) {$T_0$};
    \node[below] at (2,-.1) {$T_1$};
    \node[below] at (4,-.1) {$T_2$}; 
    \node[below] at (10,-.1) {$T_n$};
    \node at (6,0) {$\cdots$};
\end{tikzpicture}
\end{center}
In particular, $H_\psi^+$ is abstractly isomorphic to $H^+$.  

Furthermore, this isomorphism is an isomorphism of Hilbert algebras;
and, if the Haar measures on $\Spt(W)$ and $\SO(V^+)$ are respectively normalized by 
	\[
		\vol(\It) = \dim(\tau_0) = q^{en} = |2|^{-n}
			\quad \text{ and } \quad
		\vol(I^+) = 1,
	\]
then the Plancherel measures on $H_\psi^+$ and $H^+$ coincide.
\end{Thm}

\begin{proof}
We prove this theorem by investigating the structure of some 2-dimensional Hecke subalgebras.  
For $0 \le i \le n$, we define
\[
	\It_i=\It\cup \It s_i \It = \bigcap_{j \ne i} \Kt_j;
\]
In the apartment, $\It_i$ corresponds to the wall which separates the
fundamental chamber $I$ from the chamber $s_i I s_i^{-1}$ or,
equivalently, to the wall whose vertices are $K_j$ with $j \ne i$. The
rank 2 picture is as follows.

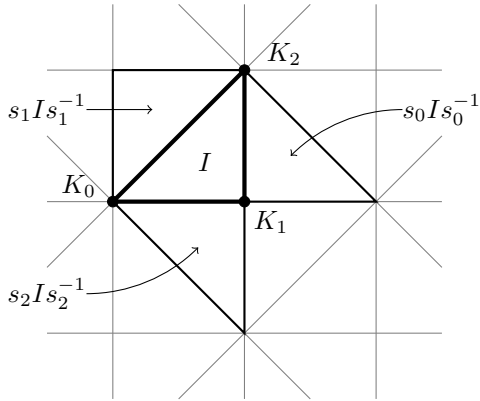
\begin{center}
\begin{tikzpicture}[scale=3.5]
	\draw [gray] (.5,-.75) -- (.5,.75);
	\draw [gray] (1,-.75) -- (1,.75);
	\draw [gray] (-.25,.5) -- (1.25,.5);
	\draw [gray] (-.25,-.5) -- (1.25,-.5);
	\draw [gray] (.25,.75) -- (1.25,-.25);
	\draw [gray] (.25,-.75) -- (1.25,.25);
	\draw [gray] (-.25,0) -- (1.25,0);
	\draw [gray] (0,-.75) -- (0,.75);
	\draw [gray] (-.25,-.25) -- (.75,.75);
	\draw [gray] (-.25,.25) -- (.75,-.75);
	\draw [thick] (.5,0) -- (.5,-.5) -- (0,0) -- (0,.5) -- (.5,.5) -- (1,0) -- (.5,0);
	\draw [fill] (0,0) circle(.02);
	\draw [fill] (.5,0) circle(.02);
	\draw [fill] (.5,.5) circle(.02);
	\draw [ultra thick] (0,0) -- (.5,0) -- (.5,.5) -- (0,0);
	\node at (-.25,-.35) {$s_2 I s_2^{-1}$};
		\draw [->] (-.1,-.35) arc (-90:-45:.6);
	\node at (-.25,.35) {$s_1 I s_1^{-1}$};
		\draw [->] (-.1,.35) -- (.15,.35);
	\node at (1.25,.35) {$s_0 I s_0^{-1}$};
		\draw [->] (1.1,.35) arc (90:135:.6);
	\node at (.35,.15) {$I$};
	\node at (-.13,.06) {$K_0$};
	\node [below right] at (.5,0) {$K_1$};
	\node at (.65,.56) {$K_2$};
\end{tikzpicture}
\end{center}

We take $H_{\psi,i}^+$ to be the subalgebra consisting of elements supported 
on $\It_i$; that is,
	\[	H_{\psi,i}^+ = \HH{\It_i}{\It}{\tau_0}.		\]
This subalgebra is at most 2-dimensional and is isomorphic to 
$\End_{\It_i} \! \big( \Ind_{\It}^{\It_i}(\tau_0^*) \big)$; it is exactly 2-dimensional 
if and only if the induced representation is reducible. 

We define $\tau_{0,i}$ to be the subspace of $S(Y)$ generated by the action of $\It_i$ on $\tau_0$;
by Lemma \ref{L:inflate},
\[
	\tau_{0,i}
	=	\begin{cases}
			\tau_0		&	\text{ if } i \ne 0	\\
			\tau_1^+	&	\text{ if } i = 0.
		\end{cases}
\]
Working 
in the dual setting, Frobenius reciprocity guarantees that $\tau_{0,i}$ may be realized as 
a submodule of $\Ind_{\It}^{\It_i}(\tau_0)$, so it suffices to verify that it is a submodule of 
strictly smaller dimension.  
We note that
\[	
	d 	=	\dim \big( \Ind_{\It}^{\It_i}(\tau_0) \big) 
		=	\dim (\tau_0) \cdot [\It_i:\It] = q^{en}(q+1)
\]
and
\[
	d_1 
	=	\dim (\tau_{0,i}) 
	=	\begin{cases}
			q^{en} 							& 	\text{if } i \neq 0,	\\
			\tfrac{1}{2} q^{en}(q + 1) 	& 	\text{if } i = 0,
		\end{cases}		
\]
hence $\Ind_{\It}^{\It_i}(\tau_0^*)$ is indeed reducible.

Since $H_{\psi,i}^+$ is 2-dimensional, it contains an element $T_i$ which is 
supported precisely on $\It s_i \It$.  In order to normalize $T_i$ and to compute its 
quadratic relation, we consider the decomposition
	\[	\Ind_{\It}^{\It_i} (\tau_0^\ast) = \pi_1^* \oplus \pi_2^*,		\]
where $\pi_1^* = \tau_{0,i}^*$ has dimension $d_1$ and $\pi_2^*$ has dimension
	\[	d_2 = d - d_1
			= 	\begin{cases}
					q^{en+1} & \text{if } i \neq 0	,  	\\
					\tfrac{1}{2} q^{en}(q+1) & \text{if } i = 0.
				\end{cases}			\] 
We normalize $T_i$ to act by $\lambda_2 = -1$ on $\pi_2^*$ and by $\lambda_1$ on $\pi_1^*$.  
Using Lemma \ref{L:Hecke}, we have
	\[	\lambda_1 = \frac{d_2}{d_1} 
				=
				\begin{cases}
						q & \text{if } i \neq 0,	\\
						1 & \text{if } i = 0,
				\end{cases}			\]
giving the desired quadratic relation $(T_i + 1)(T_i - \lambda_1) = 0$.  The invertibility 
of $T_i$ follows from its quadratic relation; explicitly, 
	\[	T_0^{-1} = T_0
			\quad \text{ and } \quad
		T_i^{-1}=q^{-1}(T_i-q+1)
			\quad \text{ for } i \neq 0.		\] 
			
Suppose that we have a braid relation
	\[
		s_i s_j \cdots \; = s_j s_i \cdots
	\]
in $\Omega_a$.  Then each of the Hecke algebra elements $T_i T_j \dots$ and $T_j T_i \dots$ is supported on the same $\It$-double coset.  From the normalization of the $T_i$, each of these elements must act on 
$(\tau_0 \otimes \tau_0^*)^{\It}$ in the same way. Whence
	\[
		T_i T_j \cdots \; = T_j T_i \cdots.
	\]  
Therefore, the braid relations for the $T_i$ are the same as those for the $s_i$, so any minimal expression $w = s_{i_1} \cdots s_{i_r}$ defines a Hecke algebra
element $T_w = T_{i_1} \cdots T_{i_r}$ supported on $\It w \It$.
From the quadratic and braid relations, we have an explicit isomorphism $H_\psi^+ \to H^+$ given by
	\[
		T_0 \mapsto \sigma
			\quad \text{ and } \quad
		T_i \mapsto U_i
			\quad \text{ for }
		i \neq 0.
	\]

We now show that $H_\psi^+ \cong H^+$ is an isomorphism of Hilbert
algebras. As each Hecke algebra is supported on its respective affine Weyl group, we have that
	\[	\operatorname{tr}(T_w)
			= 	\begin{cases}
					1 & \text{if } w = 1,		\\
					0 & \text{if } w \neq 1
				\end{cases}
			\quad \text{ and } \quad
		\operatorname{tr}(U_w)
			= 	\begin{cases}
					1 & \text{if } w = 1,		\\
					0 & \text{if } w \neq 1,
				\end{cases}			\]
so the trace-zero property is clearly preserved. For $w \in
\Omega_a^+$, the $I^+$-double cosets of $w$ and $w^{-1}$ are equal, so
the  $*$-operation in $H^+$ satisfies $U_i^* = U_i$, and hence, $U_w^* = U_{w^{-1}}$.  In 
$H_\psi^+$, we have that $T_i^*$ and $T_i$ are both supported on $\It s_i \It$, so 
$T_i^*$ acts on $\tau_0^*$ by a constant.  For $\phi \in \tau_0^*$, 
$[ \phi, T_i^* \phi ] = [ T_i \phi, \phi ]$, so $T_i^*$ and $T_i$ must act by the same constant.  
Thus, $T_i^* = T_i$ and $T_w^* = T_{w^{-1}}$, so the $*$-operation 
is also preserved.  

From the normalization given in the statement of the theorem, the
preservation of the Plancherel measures follows immediately from Lemma
\ref{L:Transfer}.
\end{proof}

\begin{Cor}
\label{C:Plancherel1}
The isomorphism
$H_\psi^+ \cong H^+$ preserves the formal degree of the Steinberg representations of the respective
Hecke algebras. 
\end{Cor}

\begin{Rmk}
If $p\neq 2$, then
  the proof for the  isomorphism $H_\psi^+ \cong H^+$ is essentially the one given in \cite{GS2}.  One notable difference is
  in the specific normalization of Hecke operators, which is always a delicate issue.
  In \cite{GS2}, they work in the central extension $\Spt(W)_8$ of $\Sp(W)$ by the 8th roots 
  of unity and normalize the generating Hecke operators to act on certain lifts of affine 
  reflections in a specified way.  We have opted to normalize the generating Hecke operator $T_i$ to 
  act by $-1$ on the irreducible representation not containing the minimal type $\tau_0^*$.
\end{Rmk}

\begin{Rmk}
The Steinberg representation of a Hecke algebra is defined by having each of the generating Hecke 
operators act by $-1$.  In \cite{GS2}, they show directly that the formal degrees 
of the respective Steinberg representations coincide.  This computation is avoided here 
because it follows from the more general coincidence of the induced Plancherel measures.  Indeed,
the implementation of the theory of induced Plancherel measures is the other notable difference
between this proof and that of \cite{GS2}.
\end{Rmk}

\begin{Rmk}
Assuming $k = \Q_2$, an isomorphic Hecke algebra is constructed in \cite{Wo}
  by finding a 1-dimensional type for a subgroup of $\It$.  This construction 
  extends to the case where $k$ is an unramified extension of $\Q_2$ but does 
  not appear to work for ramified extensions.
\end{Rmk}


\subsection{Hecke algebra of $\SO(V^-)$}


For the remainder of the section, we suppose that $n \geq 2$.  Let $V^-$ be a quadratic space 
of dimension $2n+1$ with trivial discriminant and non-trivial Hasse invariant; then $\SO(V^-)$
is the non-split inner form of $\SO(V^+)$.  Let $I^-$ be the Iwahori 
subgroup of $\SO(V^-)$, which is the pointwise stabilizer of a fundamental chamber in its 
Bruhat-Tits building, and $\Omega_a^-$ its affine Weyl group, which is generated by reflections 
$s^-_1, \dots, s^-_n$ subject to the braid relations of the affine diagram of type C$_{n-1}$.  
\begin{center}
\begin{tikzpicture}[scale=.45]
    \draw (0,-.1) --(2,-.1);
    \draw (0,.1) --(2,.1);
    \draw (2,0) --(5,0);
    \draw (7,0) --(8,0);
    \draw (8,-.1) --(10,-.1);
    \draw (8,.1) --(10,.1);
    \draw[fill] (0,0) circle(5pt);
    \draw[fill] (2,0) circle(5pt);
    \draw[fill] (4,0) circle(5pt);
    \draw[fill] (8,0) circle(5pt);
    \draw[fill] (10,0) circle(5pt);
    \node[below] at (0,-.1) {$s^-_1$};
    \node[below] at (2,-.1) {$s^-_2$};
    \node[below] at (4,-.1) {$s^-_3$}; 
    \node[below] at (10,-.1) {$s^-_n$};
    \node at (6,0) {$\cdots$};
\end{tikzpicture}
\end{center}
The standard Iwahori-Hecke algebra is the set of smooth, compactly-supported $I^-$\!-bi-invariant
functions on $\SO(V^-)$,
	\[
		H^- = \HH{\SO(V^-)}{I^-}{\mathbf{1}}.
	\]

For each $w \in \Omega^-_a$, let $U_w$
be the characteristic function on the double coset $I^- w I^-$.  The collection $\{ U_w \}$ forms a basis
of $H^-$ as a vector space.  As an algebra, $H^-$ is generated by $U_1, \dots, U_n$, where 
$U_i = U_{s^-_i}$.  These generators satisfy the quadratic relations,
	\[	(U_1 + 1)(U_1 - q^2) = 0
			\quad \text{ and } \quad
		(U_i + 1)(U_i - q) = 0
			\quad \text{ for } i \neq 1,	\] 
and the same braid relations as the $s^-_i$.  See \cite{GS2} or \cite{Ti} for details.


\subsection{$\tau_1^-$-spherical Hecke algebra of $\Spt(W)$}


We define the open compact subgroup $\Jt \subseteq \Spt(W)$ to be the full inverse image of
	\[	J = \bigcap_{j \ne 0} K_j = I \cup I s_0 I,		\]
and consider the restriction of $\tau_1^-$ to $\Jt$.  The group $\Jt$ contains the metaplectic
preimage of the subgroup
	\[	
		\Phi_{-2 \epsilon_1 + 1} \big( \SL_2(\Ok) \big) \times I_{n-1},
	\]
where $I_{n-1}$ is an Iwahori subgroup of the symplectic group of type C$_{n-1}$.  
From \cite{SW}, each component of this
direct product acts irreducibly on the corresponding component of the tensor product
	\[
		\tau_1^- = S(\Ok / 2 \varpi \Ok)^- \otimes S( \Ok^{n-1} / 2 \Ok^{n-1} ),
	\]
hence the restriction of $\tau_1^-$ to $\Jt$ must remain irreducible.
In this section, we compute the $\tau_1^-$-spherical Hecke algebra
	\[
		H_\psi^- = \HH{\Spt(W)}{\Jt}{\tau_1^-}.
	\]

We define $\Omega'_a = \la s'_1,\dots, s'_n \ra \subseteq \Omega_a$, where
\[
s'_i = \begin{cases} s_i & \text{if } i \neq 1 \\ s_1 s_0 s_1 & \text{if } i = 1. \end{cases}
\]
The reflection $s'_1$ corresponds to the affine reflection $s_{-2\epsilon_2 + 1}$, hence
$\Omega'_a$ is isomorphic to the affine Weyl group of type C$_{n-1}$, i.e., to $\Omega_a^-$;
explicitly, $\Omega'_a$ acts as the affine Weyl group of type C$_{n-1}$ on the hyperplane
$P_0$.

The proof of the following lemma is a slight variation on that of \cite[Lemma 10]{GS2}.

\begin{Lem}\label{L:support}
The support of $H_\psi^-$ is contained in $\Jt \Omega'_a \Jt$.
\end{Lem}

\begin{proof}
Fix $f \in H_\psi^-$ and $\sigma \in \Omega_a$.  Write $\sigma = a s$, where $s \in \Omega$ 
and $a$ is translation by $(a_1, \dots, a_n)$.  As $J = I \cup I s_0 I$, we have that $\sigma$ 
and $s_0 \sigma$ represent the same $\Jt$-double coset, so it suffices to show that $f(\sigma) \neq 0$ 
implies that either $\sigma$ or $s_0 \sigma$ is in $\Omega'_a$.

The element $\sigma$ conjugates the root group of $\alpha_0$ to the root group of $\beta + m$, where $\beta = s^{-1}(-2\epsilon_1)$ is a long root and $m = 1 - 2a_1$; in particular, $\sigma^{-1} \xt_{\alpha_0}(4) \sigma = \xt_{\beta+m}(t)$ for some $t \in 4 \Ok$.

From the description of the Weil representation in Section \ref{S:Weil}, we derive the following criteria for long roots $\alpha$:
	\[	\xt_\alpha(u) \in \ker \tau_1^- 
			\quad \text{ if and only if } \quad
		\begin{cases}
			u \in 4 \Ok	& \text{if } \alpha \neq -2\epsilon_1,		\\
			u \in 4 \varpi^2 \Ok & \text{if } \alpha = -2 \epsilon_1.
		\end{cases}				\]
Therefore, $\sigma$ or $s_0 \sigma$ is in $\Omega'_a$ if and only if $\beta + m \neq \pm \alpha_0$.

Let $t \in 4\Ok$ be such that
	\[	\sigma^{-1} \xt_{\alpha_0}(4) \sigma = \xt_{\beta+m}(t).		\]  
First suppose that $m > -1$ and $\beta \neq -2\epsilon_1$ or that $m > 1$ and $\beta = -2 \epsilon_1$.  Then
$\xt_{\beta+m}(t) \in \ker \tau_1^-$ and $\xt_{\alpha_0}(4) \notin \ker \tau_1^-$, 
hence
	\[	f(\sigma) 	
				= f(\sigma) (\tau_1^-)^* \big( \xt_{\beta+m}(t) \big)
				= f\big( \sigma \xt_{\beta+m}(t) \big)
				= f\big( \xt_{\alpha_0}(4) \sigma \big)
				= (\tau_1^-)^* \big( \xt_{\alpha_0}(4) \big) f(\sigma),			\]
giving that $f(\sigma) = 0$.  

Now suppose that $m < 1$ and $\beta \neq 2\epsilon_1$ or that $m < -1$ and $\beta = 2\epsilon_1$.  
Then $\xt_{-\beta-m}(t) \in \ker \tau_1^-$ 
and $\xt_{-\alpha_0}(4) \notin \ker \tau_1^-$, hence 
	\[	f(\sigma) 	
				= f(\sigma) (\tau_1^-)^* \big( \xt_{-\beta-m}(t) \big)
				= f\big( \sigma \xt_{-\beta-m}(t) \big)
				= f\big( \xt_{-\alpha_0}(4) \sigma \big)
				= (\tau_1^-)^* \big( \xt_{-\alpha_0}(4) \big) f(\sigma),			\]
giving that $f(\sigma) = 0$.  

In sum, if $f(w) \neq 0$, then $\beta+m$ must equal $\pm \alpha_0$, and the lemma is proved.
\end{proof}

\begin{Thm}\label{T:isomorphism2}
The Hecke algebra $H_\psi^-$ is generated by
invertible elements $T_1, \dots, T_n$, satisfying the quadratic relations
	\[	(T_1 + 1)(T_1 - q^2) = 0
			\quad \text{ and } \quad
		(T_n + 1)(T_n - q) = 0
			\quad \text{ for } i \neq 0,	\] 
and the braid relations of the affine diagram of type C$_{n-1}$.
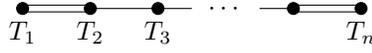
\begin{center}
\begin{tikzpicture}[scale=.45]
    \draw (0,-.1) --(2,-.1);
    \draw (0,.1) --(2,.1);
    \draw (2,0) --(5,0);
    \draw (7,0) --(8,0);
    \draw (8,-.1) --(10,-.1);
    \draw (8,.1) --(10,.1);
    \draw[fill] (0,0) circle(5pt);
    \draw[fill] (2,0) circle(5pt);
    \draw[fill] (4,0) circle(5pt);
    \draw[fill] (8,0) circle(5pt);
    \draw[fill] (10,0) circle(5pt);
    \node[below] at (0,-.1) {$T_1$};
    \node[below] at (2,-.1) {$T_2$};
    \node[below] at (4,-.1) {$T_3$}; 
    \node[below] at (10,-.1) {$T_n$};
    \node at (6,0) {$\cdots$};
\end{tikzpicture}
\end{center}
In particular, $H_\psi^-$ is abstractly isomorphic to $H^-$.

Furthermore, this isomorphism is an isomorphism of Hilbert algebras
and, if the Haar measures on $\Spt(W)$ and $\SO(V^-)$ are respectively normalized by 
	\[
		\vol(\Jt) = \dim \tau_1^- = \tfrac{1}{2} q^{en} (q - 1)
			\quad \text{ and } \quad
		\vol(I^-) = 1,
	\]
then the Plancherel measures on $H_\psi^-$ and $H^-$ coincide.
\end{Thm}

\begin{proof}
This proof is similar to the proof of Theorem \ref{T:isomorphism1}.  We investigate the structure of 
some 2-dimensional Hecke subalgebras in order to see that $H_\psi^-$ is supported 
exactly on $\Jt \Omega'_a \Jt$.  For $1 \le i \le n$, we define $\Jt_i$ to be the 
group generated by $\Jt$ and $\Jt s'_i \Jt$; in particular, $\Jt_i$ is the full inverse image of
\[
	J_i = \bigcap_{j \ne 0,i} K_j.
\]
The group $J$ corresponds to the facet of the fundamental chamber with vertices $K_1, \dots, K_n$, i.e., the facet that lies in the hyperplane $P_0$.  The conjugate $s'_i J (s'_i)^{-1}$ corresponds to a facet in the same hyperplane.  The figure on the left depicts the apartment in rank 2;  the figure on the right depicts the hyperplane $P_0$ in the rank 3 case.

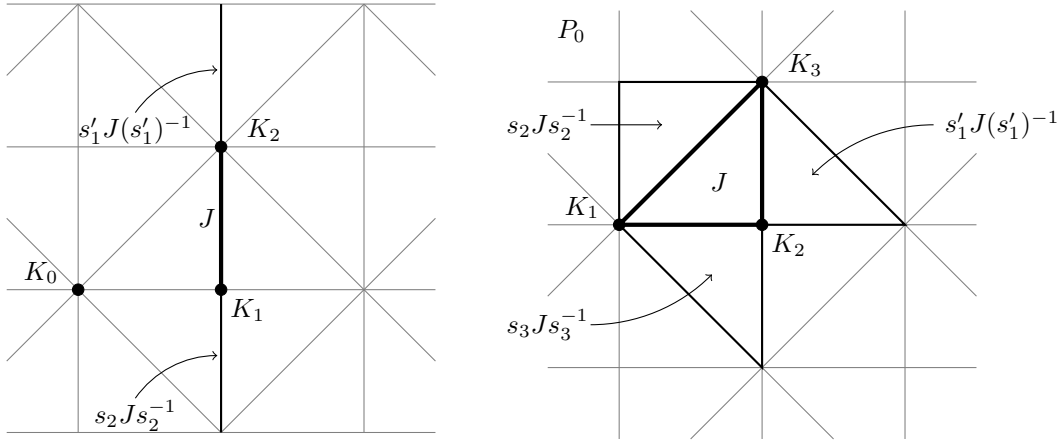
\begin{center}
\begin{tikzpicture}[scale = 3.8]
	\draw [gray] (.5,-.5) -- (.5,1);
	\draw [gray] (1,-.5) -- (1,1);
	\draw [gray] (-.25,.5) -- (1.25,.5);
	\draw [gray] (-.25,-.5) -- (1.25,-.5);
	\draw [gray] (0,1) -- (1.25,-.25);
	\draw [gray] (.5,-.5) -- (1.25,.25);
	\draw [gray] (-.25,0) -- (1.25,0);
	\draw [gray] (0,-.5) -- (0,1);
	\draw [gray] (-.25,-.25) -- (1,1);
	\draw [gray] (-.25,.25) -- (.5,-.5);
	\draw [gray] (-.25,1) -- (1.25,1);
	\draw [gray] (-.25,.75) -- (0,1);
	\draw [gray] (1,1) -- (1.25,.75);
	\draw [thick] (.5,-.5) -- (.5,1);
	\draw [ultra thick] (.5,0) -- (.5,.5);
	\draw [fill] (0,0) circle(.02);
	\draw [fill] (.5,0) circle(.02);
	\draw [fill] (.5,.5) circle(.02);
	\node at (-.13,.06) {$K_0$};
	\node [below right] at (.5,0) {$K_1$};
	\node at (.65,.56) {$K_2$};
	\node at (.45,.25) {$J$};
	\node at (.2,-.44) {$s_2 J s_2^{-1}$};
		\draw[->] (.2,-.38) arc (145:90:.35);
	\node at (.2,.56) {$s'_1 J (s'_1)^{-1}$};
		\draw[->] (.2,.62) arc (145:90:.35);
\end{tikzpicture}
\quad \quad 
\begin{tikzpicture}[scale=3.8]
	\node [below right] at (-.25,.75) {$P_0$};
	\draw [gray] (.5,-.75) -- (.5,.75);
	\draw [gray] (1,-.75) -- (1,.75);
	\draw [gray] (-.25,.5) -- (1.25,.5);
	\draw [gray] (-.25,-.5) -- (1.25,-.5);
	\draw [gray] (.25,.75) -- (1.25,-.25);
	\draw [gray] (.25,-.75) -- (1.25,.25);
	\draw [gray] (-.25,0) -- (1.25,0);
	\draw [gray] (0,-.75) -- (0,.75);
	\draw [gray] (-.25,-.25) -- (.75,.75);
	\draw [gray] (-.25,.25) -- (.75,-.75);
	\draw [thick] (.5,0) -- (.5,-.5) -- (0,0) -- (0,.5) -- (.5,.5) -- (1,0) -- (.5,0);
	\draw [fill] (0,0) circle(.02);
	\draw [fill] (.5,0) circle(.02);
	\draw [fill] (.5,.5) circle(.02);
	\draw [ultra thick] (0,0) -- (.5,0) -- (.5,.5) -- (0,0);
	\node at (-.25,-.35) {$s_3 J s_3^{-1}$};
		\draw [->] (-.1,-.35) arc (-90:-45:.6);
	\node at (-.25,.35) {$s_2 J s_2^{-1}$};
		\draw [->] (-.1,.35) -- (.15,.35);
	\node at (1.34,.35) {$s'_1 J (s'_1)^{-1}$};
		\draw [->] (1.1,.35) arc (90:135:.6);
	\node at (.35,.15) {$J$};
	\node at (-.13,.06) {$K_1$};
	\node [below right] at (.5,0) {$K_2$};
	\node at (.65,.56) {$K_3$};
\end{tikzpicture}
\end{center}

If $i \neq 1$, then $\Jt_i = \Jt \cup \Jt s_i \Jt$, hence $[\Jt_i : \Jt] = q + 1$; for the case $i = 1$, we note that $\Jt_1$ is the union of those $\It w \It$ for which $w$ is in the group generated by $s_0$ and $s_1$, hence
	\[	[ \Jt_1 : \Jt ] 
		= \frac{ [\Jt_1 : \It] }{ [\Jt : \It] }
		= \frac{ 1 + 2q + 2q^2 + 2q^3 + q^4 }{ 1 + q }
		= 1 + q + q^2 + q^3.		\]

We take $H_{\psi,i}^-$ to be the subalgebra of $H_\psi^-$ consisting of elements supported on $\Jt_i$; that is,
	\[	H_{\psi,i}^-=\HH{\Jt_i}{\Jt}{\tau_1^-}.		\]
This subalgebra is at most 2-dimensional and is isomorphic to $\End_{\Jt_i}\big( \Ind_{\Jt}^{\Jt_i} (\tau_1^-)^* \big)$. 

Let $\tau_{1,i}^-$ be the subspace of $S(Y)$ generated by the action
of $\Jt_i$ on $\tau_1^-$.  We use Lemma \ref{L:inflate} repeatedly to
compute $\tau_{1,i}$.  First, we note that $\Jt$ preserves $\tau_i^-$
for $1 \le i \le n$, so it suffices to consider the action of $s'_i$
on $\tau_1^-$.  Next, if $i \ge 2$, then $s'_i = s_i$ preserves
$\tau_1^-$, hence $\tau_{1,i}^- = \tau_1^-$.  

Lastly, we claim that $\tau_{1,1}^- = \tau_2^-$.  Since $s'_1 = s_1
s_0 s_1$, we need to consider the action of $s_1$ and $s_0$.  By the
same lemma, $s_1$ inflates $\tau_1^-$ to $\tau_2^-$.  Hence, it
remains to show that, if $\phi \in \tau_2^-$, then $s_0 \phi \in \tau_2^-$.
For such $\phi$, we recall from Section \ref{S:Weil} that the first component $\phi_1$ is in
$S(\Ok/2\varpi\Ok)^-$ and that $s_0$ acts on $\phi_1$ via
\[
	\big[ s_0 \phi_1 \big] (y_1)
	=	c \widehat\phi_1(\varpi^{-1} y_1).
\]
Since the Fourier transform maps $S(\Ok/2\varpi\Ok)^-$ to $S(\varpi^{-1}\Ok/2\Ok)^-$, the first component of $s_0\phi$ remains in $S(\Ok/2\varpi\Ok)^-$ and the claim is proved.

To see that $H_{\psi,i}^-$ is exactly 2-dimensional, we again work in the dual setting and note that
	\[ d_1 = \dim \tau_{1,i}^-
			=	\begin{cases}
					\dim (\tau_1^-) & \text{if } i \neq 1,	\phantom{\Big|}\\
					\dim (\tau_2^-) & \text{if } i = 1 \phantom{\Big|}
				\end{cases}
			=	\begin{cases}
					\tfrac{1}{2} q^{en} (q - 1) & \text{if } i \neq 1, \phantom{\Big|}	\\
					\tfrac{1}{2} q^{en} (q^2 - 1) & \text{if } i = 1 \phantom{\Big|}
				\end{cases}		\]
is strictly smaller than
	\[	d = \dim \big(\Ind_{\Jt}^{\Jt_i} (\tau_1^-) \big) 
			= \dim (\tau_1^-) \cdot [\Jt_i:\Jt]
			=	\begin{cases}
					\tfrac{1}{2} q^{en} (q^2 - 1) & \text{if } i \neq 1, \phantom{\Big|}		\\
					\tfrac{1}{2} q^{en} (q^4 - 1) & \text{if } i = 1.  \phantom{\Big|}
				\end{cases}			\]

Hence, for $1 \le i \le n$, there exists $T_i \in H_\psi^-$ supported precisely on $\Jt s_i \Jt$.  We consider the decomposition
	\[ \Ind_{\Jt}^{\Jt_i} (\tau_1^-)^* = \pi_1^* \oplus \pi_2^*,		\]
where $\pi_1^* = (\tau_{1,i}^-)^*$ has dimension $d_1$ and $\pi_2^*$ has dimension
	\[	d_2 = d - d_1 
			=	\begin{cases}
					\tfrac{1}{2}q^{en}(q^2 - q) & \text{if } i \neq 1, \phantom{\Big|}		\\
					\tfrac{1}{2}q^{en}(q^4 - q^2) & \text{if } i = 1. \phantom{\Big|}
				\end{cases}			\]
We normalize $T_i$ to act by $\lambda_2 = -1$ on $\pi_2^*$ and by $\lambda_1$ on $\pi_1^*$.  Using Lemma \ref{L:Hecke}, we have
	\[	\lambda_1 
			= \frac{d_2}{d_1}
			= \begin{cases}
					q & \text{if } i \neq 1,		\\
					q^2 & \text{if } i = 1,
				\end{cases}				\]
giving the desired quadratic relation $(T_i + 1)(T_i - \lambda_1) = 0$.  The invertibility of $T_i$ follows from its quadratic relation; explicitly, $T_i^{-1} = \lambda_1^{-1}(T - \lambda_1 + 1)$.

The proof of the braid relations mimics the proof of Theorem \ref{T:isomorphism1} 
with $\Jt$ instead of $\It$, $\tau_1^-$ instead of $\tau_0$, and $\Omega'_a$ instead 
of $\Omega_a$.  We note that computations involving $\Jt$-double cosets involve a 
weighted length function $\ell'$ on $\Omega'_a$, defined by setting $\ell'(s'_1) = 3$ 
and $\ell'(s'_i) = 1$ if $i \neq 1$.  The details of this length function are contained 
in \cite[Prop.\ 1]{GS2}; it suffices to mention here that
\begin{enum}
	\item
	$[J w J : J] = q^{\ell'(w)}$ for all $w \in \Omega'_a$.
	
	\item
	If $w_1, w_2 \in \Omega'_a$ satisfy $\ell'(w_1) + \ell'(w_2) = \ell'(w_1 w_2)$, then 
	$J w_1 J \cdot J w_2 J = J w_1 w_2 J$.
\end{enum}

For a minimal expression $w = s'_{i_1} \cdots s'_{i_r}$ in $\Omega'_a$, the braid relations in $H_\psi^-$ allow us to define a canonical Hecke operator $T_w = T_{i_1} \cdots T_{i_r}$ supported precisely on the double coset $\Jt w \Jt$.  From the quadratic and braid relations, we have the explicit isomorphism $H_\psi^- \to H^-$ given by $T_i \mapsto U_i$.

As in the proof of Theorem \ref{T:isomorphism1}, one can show that $H_\psi^- \cong H^-$ as Hilbert
algebras, hence Lemma \ref{L:Transfer} will give the coincidence of 
Plancherel measures, under the prescribed normalization.
\end{proof}

\begin{Cor}
\label{C:Plancherel2}
The isomorphism
$H_\psi^- \cong H^-$ preserves the formal degree of the Steinberg representations of the respective
Hecke algebras.
\end{Cor}

\begin{Rmk}
Let $\epsilon$ be $+$ or $-$.  Even for a fixed character $\psi$, the isomorphism $H_\psi^{\epsilon} \cong H^\epsilon$ constructed in these two sections is far from unique: the Hecke algebras admit many inner automorphisms which preserve the quadratic and braid relations of the generators.
\end{Rmk}

\begin{Rmk}
These two Hecke algebra isomorphims may be constructed using additive characters of any conductor; we chose the conductor $2e$ for its convenience.  For an even conductor, the method of construction would essentially go unchanged.  For an odd conductor, different minimal types from \cite{SW} would be employed.  What follows is a very brief summary of the relevant details.

Consider an additive character $\psi'$ given by $\psi'(t) = \psi(\varpi^{-c} t)$.
This new character has conductor $2e + c$; that is, $4 \varpi^c \Ok$ is the largest subgroup of $k$ on which $\psi'$ acts trivially.  The parity of $c$ is important, so we write $c = 2k + \delta$, where $\delta$ is 0 or 1, and we define the diagonal matrix $g_c$ to have $\varpi^{k+\delta}$ in the first $n$ entries and $\varpi^{-k}$ in the last $n$ entries.  

The conjugation, $x \mapsto x_c = g_c^{-1} x g_c$, is an automorphism of $\Sp(W)$.  If $\delta = 0$, $g_c$ is an element of the affine Weyl group and conjugation by $g_c$ is an inner automorphism.  If $\delta = 1$, $g_c$ is not an element of the affine Weyl group (or even of the symplectic group!) and conjugation by $g_c$ is an outer automorphism.  We write $G_c = g_c^{-1} G g_c$ for any subgroup $G$ of $\Sp(W)$ and $\widetilde{G}_c$ for its inverse image in $\Spt(W)$.  In the apartment, conjugation by $g_c$ corresponds to translation by $c z_n$, so the fundamental chamber $I$ with vertices $K_0, \dots, K_n$ is translated to the chamber $I_c$ with vertices $K_{0,c}, \dots, K_{n,c}$.  If $\delta = 0$, then $K_{i,c}$ is conjugate in $\Sp(W)$ to $K_i$.  If $\delta = 1$, then $K_{i,c}$ is conjugate in $\Sp(W)$ to $K_{n-i}$.  For any $\delta$, $I_c$ is conjugate in $\Sp(W)$ to $I$.  The rank 2 picture is as follows.

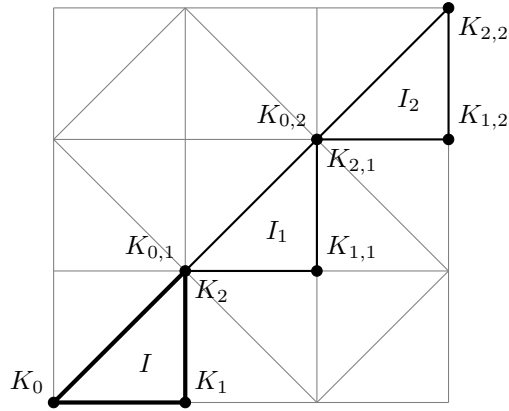
\begin{center}
\begin{tikzpicture}[scale=3.5]
	\draw [gray] (0,0) -- (0,1.5);
	\draw [gray] (.5,0) -- (.5,1.5);
	\draw [gray] (1,0) -- (1,1.5);
	\draw [gray] (1.5,0) -- (1.5,1.5);
	\draw [gray] (0,0) -- (1.5,0);
	\draw [gray] (0,.5) -- (1.5,.5);
	\draw [gray] (0,1) -- (1.5,1);
	\draw [gray] (0,1.5) -- (1.5,1.5);
	\draw [gray] (0,0) -- (1.5,1.5);
	\draw [gray] (0,1) -- (.5,1.5);
	\draw [gray] (1,0) -- (1.5,.5);
	\draw [gray] (0,1) -- (1,0);
	\draw [gray] (.5,1.5) -- (1.5,.5);
	\draw [fill] (0,0) circle(.02);
	\draw [fill] (.5,0) circle(.02);
	\draw [fill] (.5,.5) circle(.02);
	\draw [fill] (1,.5) circle(.02);
	\draw [fill] (1,1) circle(.02);
	\draw [fill] (1.5,1) circle(.02);
	\draw [fill] (1.5,1.5) circle(.02);
	\draw [ultra thick] (0,0) -- (.5,0) -- (.5,.5) -- (0,0);
	\draw [thick] (.5,.5) -- (1,.5) -- (1,1) -- (.5,.5);
	\draw [thick] (1,1) -- (1.5,1) -- (1.5,1.5) -- (1,1);
	\node at (.35,.15) {$I$};
	\node at (.85,.65) {$I_1$};
	\node at (1.35,1.15) {$I_2$};
	\node [above left] at (0,0) {$K_0$};
	\node [below right] at (.5,.5) {$K_2$};
	\node [above right] at (.5,0) {$K_1$};
	\node [above left] at (.5,.5) {$K_{0,1}$};
	\node [below right] at (1,1) {$K_{2,1}$};
	\node [above right] at (1,.5) {$K_{1,1}$};
	\node [above left] at (1,1) {$K_{0,2}$};
	\node [below right] at (1.5,1.5) {$K_{2,2}$};
	\node [above right] at (1.5,1) {$K_{1,2}$};
\end{tikzpicture}
\end{center}    

If we follow conjugation by the Weil representation $\omega_\psi$, we get a Weil representation $\omega_{\psi,c}$; in particular,
\begin{align*}
	\omega_{\psi,c}\big( \xt(a) \big)	&	=	\omega_\psi\big( \xt( \varpi^{-c} a ) \big),		\\
	\omega_{\psi,c}\big( \htt(a) \big)	&	=	\omega_\psi\big( \htt(a) \big),					\\
	\omega_{\psi,c}\big( \wt \big)		&	=	\omega_\psi\big( \wt(\varpi^c) \big),
\end{align*}
where $\wt(\varpi^c)$ is a lift of $g_c^{-1} w g_c$.  Some straight-forward computations reveal that $\omega_{\psi,c} = \omega_{\psi'}$.
\begin{center}
\begin{tikzpicture}[scale=.4]
	\node at (0,0) {$\Spt(W)$};
	\node at (10,0) {$\Spt(W)$};
	\node at (5,-3) {$\GL(S(Y))$};
	\draw [->] (2,0) -- (8,0);
		\node [above] at (5,0) {$g_c$};
	\draw [->] (.5,-.8) -- (3,-2.5);
		\node at (1,-1.8) {$\omega_{\psi'}$};
	\draw [->] (9.5,-.8) -- (7,-2.5);
		\node at (9.4,-1.8) {$\omega_{\psi,c}$};
\end{tikzpicture}
\end{center}

We note that $x \in \Sp(W)$ stabilizes a lattice $\LL$ if and only if $x_c$ stabilizes $g_c^{-1} \LL$.  For convenience, we write $\LL_{i,c} = g_c^{-1} \LL_i$ and $L_{i,c} = \LL_{i,c} \cap Y$.

For $\delta = 0$, the space $\tau_{i,c} = S( L_{0,c} / 2 L_{i,c} )$ is a type for $\Kt_{i,c}$.  It is easily checked that $\tau_{i,c}$ is isomorphic to $\tau_i$.  The group $J_c = K_{1,c} \cap \dots \cap K_{n,c}$ is isomorphic to $J = K_1 \cap \dots \cap K_n$.

For $\delta = 1$, the space $\tau_{i,c} = S( L_{i,c} / 2 \varpi L_{0,c} )$ is a type for $\Kt_{i,c}$.  It is easily checked that $\tau_{i,c}$ is isomorphic to $\tau_{n-i}$.  The group $J_c = K_{1,c} \cap \dots \cap K_{n,c}$ is isomorphic to $K_0 \cap \dots \cap K_{n-1}$.

In either case, we build the two Hecke algebras
\[
	H_{\psi,c}^+ = \HH{\Spt(W)}{\It_c}{\tau_{0,c}}
		\quad \text{ and } \quad
	H_{\psi,c}^- = \HH{\Spt(W)}{\Jt_c}{\tau_{1,c}^-}.
\]
The same sort of geometry of the apartment employed in the previous sections will yield the existence of generators of $H_{\psi,c}^\pm$ with the same quadratic and braid relations as $H_\psi^\pm$.

The curious reader is referred to \cite{GS2}, where Gan and Savin use an odd conductor in their computation of $H_\psi^-$ under the assumption that $p \neq 2$.
\end{Rmk}


\section{\bf Equivalence of categories between $\G_\psi^\pm$ and
  $\SS_0^\pm$}


In the category of smooth genuine representations of $\Spt(W)$, let $\G_\psi^\pm$ be the
Bernstein component containing the even/odd Weil representation $\omega_\psi^\pm$.  
In the category of smooth representations of $\SO(V^\pm)$, let $\SS_0^\pm$ be the 
Bernstein component containing the trivial representation.

We will prove our main theorem, namely that there is an
equivalence of categories between $\G_\psi^\epsilon$ and $\SS_0^\epsilon$, where $\epsilon$ is $+$ or $-$.
Our proof essentially follows that of \cite{GS2}.


\subsection{\bf Equivalence between $\G_\psi^+$ and
  $\SS_0^+$}


Let $U$ (resp. $U^-$) be the unipotent radical in $\Sp(W)$ generated by positive (resp. negative) root groups.  Let $\Bt=\Tt
U\subseteq\Spt(W)$ be the preimage of the Borel subgroup $B=TU$ of
$\Sp(W)$. (Recall that the unipotent radical $U$ splits in $\Spt(W)$.)

An element $t$ of the maximal torus $T$ may be expressed uniquely as
	\[
		t = (t_1, \dots, t_n) = h_{2\epsilon_1}(t_1) \cdots h_{2\epsilon_n}(t_n),
	\]
hence we have a canonical lift of $t$ given by 
	\[
		\tt = \htt_{2\epsilon_1}(t_1) \cdots \htt_{2\epsilon_n}(t_n).
	\]
With this convention, multiplication in $\Tt$ is given by
	\[
		\tt \cdot \uu 
				= (t,u) \; \widetilde{t u} 
				= \prod_{i = 1}^n (t_i, u_i) \; \htt_{2\epsilon_i}(t_i u_i),
	\]
where the cocycle $(t,u) \in \{\pm1\}$ is the product of Hilbert symbols $(t_i, u_i)$ on $k$.  Note that multiplication in $T$ is commutative.

Recalling the action of $T$ on $Y$ by $ty = (t_1 y_1, \dots, t_n y_n)$, the action of $\tt$ on $S(Y)$ is given by
	\[
		\tt \phi(y) = \beta_t |\det t|^{^1\!\!/\!_2} \phi(ty),
	\]
where $\beta_t$ is a 4th root of unity satisfying $\beta_t \beta_u = (t,u) \beta_{tu}$.

Given a character $\chi = (\chi_1, \dots, \chi_n)$ on $T$, we define a
genuine character $\chit$ on $\Tt$ by
	\[
		\chit(\,\tt\,) = \chi(t) \beta_t.
	\]
We extend this character trivially to all of $\Bt$ and define $I(\chit)$ to be the normalized 
induced representation $\Ind_{\Bt}^{\Spt(W)}\chit$.  By Frobenius reciprocity,
\begin{equation}\label{E:Frobenius}
\Hom_{\Spt(W)}(\pi, I(\chit))\cong\Hom_{\Tt}(\pi_U, \chit),
\end{equation}
where $\pi$ is any smooth representation of $\Spt(W)$ and $\pi_U$ is
the normalized Jacquet module with respect to the Borel $\Bt$.

\begin{Lem}
The Bernstein component $\G_\psi^+$ is precisely the component whose irreducible
representations are submodules of $I(\chit)$ for some unramified character $\chi$.
\end{Lem}

\begin{proof}
The functional $l:S(Y)^+\rightarrow\C$ defined by $l(\phi)=\phi(0)$
factors through the Jacquet module $(\omega_\psi^+)_U$ and gives a
non-trivial element in $\Hom_{\Tt}((\omega_\psi^+)_U, \chit)$ for some
unramified $\chi$, which in turn gives an embedding $\omega_\psi^+\subseteq
I(\chit)$ via Frobenius reciprocity.
\end{proof}

The Iwahori subgroup $\It$ admits a factorizatioin 
	\[
		\It = I_{U^-} \It_{T} I_U,
	\]
where $I_{U^-} = \It \cap U^-$, $\It_{T} = \It \cap \Tt$, and $I_U = \It \cap U$.  (Note that $I_U$ and $I_{U^-}$ split in the central extension $\Spt(W)$.)

Now let us define the ``Jacquet module'' $(\tau_0)_U$ of $\tau_0$ with respect to $I_U$;
that is, $(\tau_0)_U$ is the quotient of $\tau_0 = S(L_0 / 2L_0)$ by
	\[
		\big\la \tau_0(u)\phi-\phi : u \in I_U, \phi \in \tau_0 \big\ra,
	\]
which may be viewed as a representation of $\It_T$. 

\begin{Lem}\label{L:Jacquet_tau_0}
The space $(\tau_0)_U$ is one dimensional and spanned by the image of the
characteristic function of $2L_0$. Moreover each element
$\tt \in \It_T$ acts by $\beta_t$ on $(\tau_0)_U$.
\end{Lem}

\begin{proof}
Suppose that $\phi \in \tau_0$ is supported on $a + 2L_0$ for $a \in L_0 \smallsetminus 2L_0$, and 
let $i$ be such that $a_i \in \Ok^\times$.  The element $\xt_{2 \epsilon_i}(1)$ acts on $\phi$ by the 
constant $\psi(a_i^2) \neq 1$.  Therefore, the image of $\phi$ in $(\tau_0)_U$ is trivial.

On the other hand, let $\phi$ be the characteristic function on $2L_0$.  Using the formulas in Section \ref{S:Weil} for the action of the positive root groups, it is simple to check that $I_U$ acts trivially on $\phi$.  Moreover, we know that $\tt \in \It_T$ acts by
$\tt \phi(y) = \beta_t \phi(ty) = \beta_t \phi(y)$.
\end{proof}

\begin{Thm}\label{T:equivalence}
The functor from the category $\G_\psi^+$ to the category of $H_\psi^+$-modules, given by
	\[
		\pi \mapsto (\pi \otimes \tau_0^*)^{\It},
	\]
is an equivalence of categories.   In particular, there is an equivalence of categories 
between $\G_\psi^+$ and $\SS_0^+$
given by the isomorphism $H_\psi^+\cong H^+$ of Hecke algebras.
Furthermore, this equivalence preserves the temperedness and square integrability of representations.
\end{Thm}

\begin{proof}
We have the natural surjection
\[
r:(\pi\otimes\tau_0^*)^{\It}\rightarrow (\pi_U\otimes(\tau^*_0)_U)^{\It_T},
\]
which is a slight variant of what is called ``Jacquet's Lemma'' in
\cite[64-65]{B2}. (To prove our version, one can follow the argument
there. Also, see \cite[Prop. 3.5.2]{B1}.)

We first show that $r$ is an isomorphism.  Suppose that $v \in \ker r$, i.e., that there exists an open compact subgroup $U_v$ of $U$
such that $\int_{U_v} \pi(u) v \, du = 0$.  

For a translation $\lambda = (\lambda_1, \dots, \lambda_n)$
in $D \subseteq \Omega_a$, we write 
	\[
		\lambda = \htt_{2 \epsilon_1}(\varpi^{\lambda_1}) \cdots
						\htt_{2 \epsilon_n}(\varpi^{\lambda_n})
	\] 
as its representative in $\Tt$.  Take $\lambda \in D$ such that 
$\lambda_1 \ge \ldots \ge \lambda_n \ge 0$
and $\lambda^{-1} I_U \lambda \supseteq U_v$.  Then
	\[
		\It \lambda \It = \bigcup_{i = 1}^{q^{\ell(\lambda)}} \lambda u_i \It
	\]
where the $u_i$ are representatives of the $\It_U$-cosets in 
$\lambda^{-1} I_U \lambda$.  Let $T_\lambda$ be the 
Hecke algebra element supported on $\It \lambda \It$ obtained using a minimal expression 
for $\lambda \in \Omega_a$ as in Theorem \ref{T:isomorphism1}.  Then
	\[
		\pi(T_\lambda) v 
			= \pi(\lambda) \sum_{i = 1}^{q^{\ell(\lambda)}} \pi(u_i)v
			= \pi(\lambda) \int_{\lambda^{-1}\It_U \lambda} \pi(u) v \, du
			= 0.
	\]
The element $T_\lambda$ is invertible, as it is the product of invertible elements, 
hence $v = 0$ and $r$ is injective.

Let $\pi$ be an irreducible representation of $\Spt(W)$ such that
$(\pi\otimes\tau_0^*)^{\It}\neq 0$.  As $r$ is an isomorphism,
$(\pi_U\otimes(\tau^*_0)_U)^{\It_T}\neq 0$.  This implies that
$\Hom_{\Tt}(\pi_U, \chit) \neq 0$ for some unramified $\chi$, 
since $\tt \in \It_T$ acts by $\beta_t$ on $(\tau_0)_U$.  Therefore,
by Frobenius reciprocity, we have that $\pi$ is a subrepresentation
of $I(\chit)$.

Conversely, let $\pi$ be an irreducible submodule of $I(\chit)$ for some
unramified $\chi$.  By Frobenius reciprocity, we have that
	\[
		0 \neq (\pi_U \otimes (\tau_0^*)_U)^{\It_T}
			\cong (\pi \otimes \tau_0^*)^{\It}.
	\]
Thus, condition (iii) of \cite[3.11]{BK} is satisified, which proves
the equivalence of categories. 

To complete the proof, we note that, as categories,
	\[
		\SS_0^+ 
			\cong H^+\!\text{-modules}
			\cong H_\psi^+\!\text{-modules}
			\cong \G_\psi^+;
	\] 
moreover, the trivial representation of $\SO(V^+)$ corresponds to the trivial module of 
$H^+ \cong H_\psi^+$, and hence to the even Weil represntation $\omega_\psi^+$.

Finally, to show that the equivalence preserves temperedness and
square integrability, let us note that the equivalence
$H_\psi^+\!\text{-modules}\cong \G_\psi^+$ implies that $(\It,
\tau_0)$ is an $\mathfrak{s}$-type in the sense of
\cite[1.6]{BHK}. (To see this, let $e\in H_{\psi}^+$ be the idempotent corresponding
to $\tau_0$ and $\mathfrak{R}_e(\Spt(W))$ the full subcategory of
the category of smooth genuine representations of $\Spt(W)$ as defined
in \cite[1.4]{BHK}. Then we have the functor
$\mathfrak{R}_e(\Spt(W))\rightarrow  H_\psi^+\!\text{-modules}$ given
by $\pi\mapsto (\pi\otimes\tau_0^*)^{\It}$. Clearly this functor
composed with the equivalence $H_\psi^+\!\text{-modules}\cong
\G_\psi^+$ is the identity, which implies
$\mathfrak{R}_e(\Spt(W))=\G_\psi^+$. Hence $(\It,
\tau_0)$ is an $\mathfrak{s}$-type with $\mathfrak{s}$ being the inertial
equivalence class representanted by $(\Bt, \mathbf{1})$.) Hence by
the first paragraph of \cite[0.6]{BHK}, one can see that all the
irreducible tempered representations in $\G_\psi^+$ are in $_r\widehat{G}(\tau_0)$
with $G=\Spt(W)$. The same applies to the group $\SO(V^+)$. Furthermore,
from \cite[5.1]{BHK}, the equivalence $\SS_0^+\cong
\G_\psi^+$ restricts to the homeomorphism $\widehat{\alpha}$ of Lemma
\ref{L:Transfer} with $G_1=\Spt(W), \sigma_1=\tau_0, G_2=\SO(V^+),
\sigma_2=\mathbf{1}$. Hence, we see that the equivalence preserves
temperedness and square integrability. 
\end{proof}

\begin{Rmk}
In \cite{GS2}, the preservation of temperedness and square
integrability is shown by using Casselman's criterion.  In this paper,
however, we invoke the theory of \cite{BHK}, which can be applied 
once the Hecke algebra isomorphism $H_\psi^+ \cong H^+$ is shown to be
an isomorphism of Hilbert algebras. Indeed, this is one of the benefits
of showing that $H_\psi^+ \cong H^+$ is not just an
algebra isomorphism but a Hilbert algebra isomorphism. 
\end{Rmk}


\subsection{\bf Equivalence between $\G_\psi^-$ and
  $\SS_0^-$}


Consider the partial flag
	\[
		X_n \subseteq X_{n-1} \subseteq \dots \subseteq X_2
	\]
where $X_i$ is the $k$-span of $\e_i, \dots, \e_n$.  Let $P = MN$ be the
parabolic subgroup which is the stabilizer of this partial flag.  Let
$W_1$ be the symplectic subspace spanned by $\{ \e_1, \f_1 \}$ so that
$\Sp(W_1) = \SL_2(k)$.  Define $\omega_{\psi,1}$ to be the Weil representation
of $\Spt(W_1)$, realized as a representation in the space $S(k \f_1)$.  
This representation decomposes
into even and odd parts; the odd part $\omega_{\psi,1}^-$ is supercuspidal.

Let $\Pt=\Mt N$ be the preimage of $P$ in $\Spt(W)$. Each element
$m\in\Mt$ is uniquely written as
	\[
		m = m_1 \cdot \htt_{2\epsilon_2}(t_2)\cdots \htt_{2\epsilon_n}(t_n)
	\]
where $t_i\in k^\times$ and $m_1\in\Spt(W_1)$. Given a
character $\chi = (\chi_2,\dots,\chi_n)$, we define a
genuine representation $\omega_{\psi, 1}^- \otimes \chit$ of $\Mt$ by
\[
	\big[ \omega_{\psi, 1}^- \otimes \chit \big](m)
	=	\omega_{\psi, 1}^-(m_1) \prod_{i = 2}^n \chi_i(t_i) \beta_{t_i}.
\]
We set
\[
I(\omega_{\psi, 1}^- \otimes \chit)
=\Ind_{\Pt}^{\Spt(W)}(\omega_{\psi, 1}^- \otimes \chit)
\]
to be the normalized induced representation. For a smooth representation $\pi$ of $\Spt(W)$ and 
$\pi_N$ its normalized Jacquet module, we have Frobenius reciprocity:
\[
\Hom_{\Spt(W)}(\pi, I(\omega_{\psi, 1}^- \otimes \chit))
\cong\Hom_{\Mt}(\pi_N, \omega_{\psi, 1}^- \otimes \chit).
\]

\begin{Lem}
The Bernstein component $\G_\psi^-$ is precisely the component whose irreducible
representations are submodules of $I(\omega_{\psi, 1}^- \otimes \chit)$
for an unramified character $\chi$.
\end{Lem}
\begin{proof}
The functional $l:S(Y)^-\rightarrow S(k\f_1)^-$, defined by
restriction of functions from $Y$ to $k\f_1$,
factors through the Jacquet module $(\omega_\psi^-)_N$.  Therefore, there is a
non-trivial element in $\Hom_{\Mt}((\omega_\psi^-)_N,
\omega_{\psi, 1}^- \otimes \chit)$ for some unramified $\chi$, which
gives an embedding $\omega_\psi^-\subseteq I(\omega_\psi^-\otimes\chit)$ via Frobenius reciprocity.
\end{proof}

\begin{Thm}\label{T:equivalence}
The functor from the category $\G_\psi^-$ to the category of $H_\psi^-$-modules, given by
	\[
		\pi \mapsto (\pi \otimes (\tau_1^-)^*)^{\Jt},
	\]
is an equivalence of categories.   In particular, there is an equivalence of
categories between $\G_\psi^-$ and $\SS_0^-$ given by the isomorphism
$H_\psi^- \cong H^-$ of Hecke algebras.
Furthermore, this equivalence preserves the temperedness and square integrability of representations.
\end{Thm}

\begin{proof}
Let $\Jt_M=\Jt\cap\Mt$ and $\Jt_N = \Jt\cap N$.  As in the previous subsection, 
we define the ``Jacquet module'' $(\tau_1^-)_N$ with respect to $\Jt_N$, which
we view as a representation of $\Jt_M$.  Recall that
	\[
		\tau_1^- = S(\Ok/2\varpi\Ok)^- \otimes S(\Ok^{n-1}/2\Ok^{n-1}),
	\]
and hence, 
	\[
		(\tau_1^-)_N = S(\Ok/2\varpi\Ok)^- \otimes (\tau_{0,n-1})_{U_{n-1}},
	\]
where the second factor is the Jacquet module from the previous subsection in rank $n-1$.
Therefore, $(\tau_1^-)_N$ is an irreducible representation of $\Jt_M$.

We have the natural surjection
\[
r:(\pi\otimes(\tau_1^-)^*)^{\Jt}\rightarrow (\pi_N\otimes (\tau_1^-)_N^*)^{\Jt_M}.
\]
Just as in Proposition \ref{T:equivalence}, one can show
that $r$ is injective, which together with Frobenius reciprocity
shows that $(\pi\otimes(\tau_1^-)^*)^{\Jt}\neq 0$
if and only if $\pi$ is a submodule of $I(\omega_{\psi,1}^- \otimes \chit)$
for some unramified $\chi$. Hence
\cite[(3.11)]{BK} implies the equivalence of the categories.

Finally, this equivalence implies that $(\Jt, \tau_1^-)$ is
an $\mathfrak{s}$-type in the sense of \cite{BHK}, from which one can
deduce the preservation of temperedness and square integrability just as in the previous section.
\end{proof}




\begin{Rmk}
As a final remark, let us mention that in \cite[Sec. 15 and 16]{GS2} it is shown that
the theta correspondence preserves unramified Langlands parameters,
which relies on another work \cite{GS1} of Gan and Savin. The only
obstruction to remove the $p \neq 2$ assumption from \cite{GS1}, however,
is the Howe duality conjecture, which was recently proven by Gan and the first-named author in
\cite{GT} for the case at hand. Hence, everything discussed in \cite[Sec. 15 and
16]{GS2} holds without the assumption $p\neq 2$.
\end{Rmk}

\end{document}